\documentclass[fontsize,12pt,times]{elsarticle}
\usepackage[left=2cm,right=2cm,top=3cm,bottom=3cm]{geometry}
\usepackage{amssymb}
\usepackage{amsthm}
\usepackage{latexsym}
\usepackage{amsmath}
\usepackage{xcolor}
\usepackage{graphicx}
\usepackage{indentfirst}
\usepackage{mathrsfs}
\usepackage{pdfsync}
\usepackage{hyperref}
\hypersetup{hypertex=true,
	colorlinks=true,
	linkcolor=blue,
	anchorcolor=blue,
	citecolor=blue}
\allowdisplaybreaks
\usepackage[utf8]{inputenc}

\renewcommand{\thefootnote}{\arabic{footnote}}
\newtheorem{theorem}{\color{black}\indent Theorem}
\newtheorem{lemma}{\color{black}\indent Lemma}[section]
\newtheorem{proposition}{\color{black}\indent Proposition}

\newtheorem{remark}{\color{black}\indent Remark}[section]

\newcommand\blfootnote[1]{%
		\begingroup
		\renewcommand\thefootnote{}\footnote{#1}%
		\addtocounter{footnote}{-1}%
		\endgroup
	}
\journal{?}
\begin{document}
\begin{frontmatter}
\title{KAM theory at the Quantum resonance}
\author{{ \blfootnote{$^{*}$Corresponding author } Huanhuan Yuan$^{a}$ \footnote{ E-mail address : yuanhh128@nenu.edu.cn},~ Yong Li$^{b,c,*}$  \footnote{E-mail address : liyong@jlu.edu.cn}\\
    {$^{a}$School of Mathematics and Statistics, Northeast Normal University,} {Changchun 130024, P. R. China.}\\
	{$^{b}$College of Mathematics, Jilin University,} {Changchun 130012, P. R. China.}\\
	{$^{c}$School of Mathematics and Statistics, and Center for Mathematics and Interdisciplinary Sciences, Northeast Normal University,}
	{Changchun 130024, P. R. China.}
}}

\begin{abstract}
  We consider the semiclassical operator $\hat{H}(\epsilon,h):=H_{0}(hD_{x})+\epsilon \tilde{P}_{0}$ on $L^{2}(\mathbb{R}^{l})$, where the symbol of $\hat{H}(\epsilon,h)$ corresponds to a perturbed classical Hamiltonian of the form:
  \begin{align*}
  H(x,y,\epsilon)=H_{0}(y)+\epsilon P_{0}(x,y).
  \end{align*}
  Here, $\tilde{P}_{0}=Op_{h}^{W}(P_{0})$ is a bounded pseudodifferential operator with a holomorphic symbol that decays to zero at infinity, and $\epsilon\in \mathbb{R}$ is a small parameter. We establish that for small $|\epsilon|<\epsilon^{*}$, there exists a frequency $\omega(\epsilon)$ satisfying condition \eqref{b}, such that the spectrum of $\hat{H}(\epsilon,h)$ is given by the quantization formula:
  \begin{align*}
  E(n_{y},E_{u},E_{v},\epsilon,h)=\varepsilon(h,\epsilon)+h\sum_{j=1}^{d}\omega_{j}(n_{y}^{j}+\frac{\vartheta_{j}}{4})+\frac{\epsilon}{2}\bigg(\sum_{j=1}^{d_{0}}\lambda_{j}
  (n_{u}^{j}+\frac{1}{2})+\sum_{j=1}^{d_{0}}\tilde{\lambda}_{j}(n_{v}^{j}+\frac{1}{2})\bigg)+O(\epsilon\exp(-ch^{\frac{1}{\alpha-1}})),
  \end{align*}
  where $\alpha>1$ is Gevrey index, $\vartheta$ represents the Maslov index of the torus.
  This spectral expression captures the detailed structure of the perturbed system, reflecting the influence of partial resonances in the classical dynamics. In particular, the resonance-induced quadratic terms give rise to clustering of eigenvalues, determined by the eigenvalues $\lambda_{j}$ and $\tilde{\lambda}_{j}$ of the associated quadratic form in the resonant variables. Moreover, the corresponding eigenfunctions exhibit semiclassical localization—quantum scarring—on lower-dimensional invariant tori formed via partial splitting under resonance.
\end{abstract}
\end{frontmatter}


\section{Introduction}
The persistence of quasi-invariant tori in quantum systems, as governed by quantum KAM theory, remains a central problem in the spectral analysis of the perturbed Hamiltonians. The persistence of quasi-invariant structures under small perturbations has been extensively studied in infinite-dimensional classical Hamiltonian systems, such as nonlinear wave and Schrödinger equations (e.g., Berti and Bolle, \cite{MR4321994}; Eliasson and Kuksin, \cite{MR2680422}). In the quantum setting, recent work by Grébert and Paturel, \cite{MR4093983} has addressed the reducibility and spectral properties of time-quasiperiodic perturbations of quantum harmonic oscillators using pseudodifferential techniques. The microlocal frameworks employed by Dyatlov and Zworski \cite{MR3969938} and by Sjöstrand \cite{MR2674764} have deepened the understanding of spectral instabilities and resonance distributions for non-selfadjoint and semiclassical operators. However, in systems with low-dimensional invariant tori, quantum states may exhibit anomalous localization phenomena (such as 'scarring'), deviating from classical ergodicity (Anantharaman and Macià, \cite{MR3226742}). In this study, we use rigorous microlocal analysis to show that the quantum normal form under resonance conditions induces spectral cluster splitting and supports semiclassical localization of scarring states on invariant tori. Our results offer a new perspective on spectral stability and localization in the semiclassical regime of infinite-dimensional Hamiltonian systems.

In classical Hamiltonian systems, an $l$-dimensional invariant torus may break into lower-dimensional invariant tori $\mathbb{T}^{d}$ under resonance frequency conditions. In this study, we will observe this phenomenon in quantum systems by employing the Quantum Normal Form of the Hamiltonian operator $\hat{H}(\epsilon,h)$ to analyze its spectral structure. Furthermore, we investigate the semiclassical behavior in stable directions of the corresponding quantum states on $\mathbb{T}^{d}$, with a particular focus on localization properties and spectral clustering phenomena. Inspired by Heller’s approach, we use a single-parameter control mechanism to regulate spectral density, providing insights into the distribution of quantum states on lower-dimensional invariant tori. This work not only deepens our understanding of quantum states under resonance conditions but also offers a novel perspective on quantum manifestations of KAM dynamics.

Recent advances in the classical normalization of resonant systems, such as the normalization flow (a continuous dynamical process) developed by Treschev \cite{MR4864461}, underscore the importance of resonance in shaping the structural properties of Hamiltonians. In our work, the clustering of eigenvalues induced by resonance is captured in the quantum normal form, particularly through the second-order term involving the resonant matrix $M_{p}(\omega)$, which governs the leading spectral shifts. The presence of $d_{0}$-dimensional resonances results in a finer clustering of eigenvalues within $O(\epsilon h)$-bands, with the inter-band separation determined by the unperturbed frequencies $\omega_{j}$ and the intra-band structure determined by the eigenvalues of $M_{p}(\omega)$. This clustering of eigenvalues on the $O(\epsilon h)$-scale arises from the resonant quadratic terms in the normal form, independent of the exponentially small remainder, which only influences higher-order precision in the spectral approximation. This method applies to a wide class of semiclassical pseudodifferential operators exhibiting resonance phenomena and provides a detailed understanding of spectral shifts in the presence of degeneracies.

Our approach combines semiclassical Birkhoff normal form theory with microlocal analysis to construct a quantum normal form for resonant Hamiltonians. We derive the resonant quantum normal form for semiclassical pseudodifferential operators under partial frequency commensurability and provide a full asymptotic expansion of the eigenvalues. This framework reveals resonance-induced clustering and lays the foundation for studying quantum scarring on invariant tori.

In recent years, there has been substantial progress in understanding spectral structures of semiclassical operators under perturbations. Works such as those by Kordyukov \cite{b}, Golinskii and Stepanenko \cite{MR4707547}, Bachmann et al. \cite{MR4707549}, and Beckus and Bellissard \cite{MR3568021} have explored spectral asymptotics, eigenvalue clustering, and continuity of spectral, offering valuable tools and perspectives relevant to our study of quantum normal forms and spectral quantization under resonant conditions.
Charles, L. in \cite{MR4776289} investigated the spectral structure of the magnetic Laplacian operator of Hermitian line bundles with non-degenerate curvature on compact Riemann manifolds. The author showed that at high tensor powers, the spectrum exhibits gaps and clusters, and the number of eigenvalues in each cluster was given by the Riemann-Roch number. For general results on spectral theory, see Kato \cite{MR203473} or Shubin \cite{MR883081}.
\subsection{Assumption}
We begin by defining a function space with Gevrey index $\alpha$ (where $\alpha=1$ corresponds to the setting in \cite{MR1724855}), denoted as $\mathcal{F}_{\rho,\sigma}^{\alpha}, \rho,\sigma\geq1$ and $\alpha>1$:
\begin{align*}
\mathcal{F}_{\rho,\sigma}^{\alpha}:=\{f: \mathbb{R}^{2l}\rightarrow \mathbb{C}\ |\ \|f\|_{\rho,\sigma}^{\alpha}<+\infty\}
\end{align*}
equipped with the following norm
\begin{align}\label{aw}
\|f\|_{\rho,\sigma}^{\alpha}:=\sum_{k\in\mathbb{Z}^{l}}e^{\rho|k|^{\frac{1}{\alpha}}}\int_{\mathbb{R}^{2l}}|\hat{\tilde{f}}(s)|e^{\sigma|s|^{\frac{1}{\alpha}}}{\rm d}s.
\end{align}
The class of semiclassical Weyl pseudodifferential operator $F$ in $L^{2}(\mathbb{R}^{l})$ with symbol $f(x,\xi)\in\mathcal{F}_{\rho,\sigma}^{\alpha}$ denoted $\Phi_{\rho,\sigma}^{\alpha}$ reads:
\begin{align*}
(Fu)(x):&=Op_{h}(f)u(x)\\
& =\frac{1}{(2\pi h)^{l}}\int\int_{\mathbb{R}^{l}\times\mathbb{R}^{l}}e^{i\langle (x-y),\xi\rangle/h}f\bigg(\frac{x+y}{2},\xi\bigg)u(y){\rm d}y{\rm d}\xi,\ u\in \mathcal{S}(\mathbb{R}^{l}).
\end{align*}
\begin{remark}
$F:=Op_{h}(f)$ is a trace-class operator, and for real-valued $f$, self-adjoint $h-$pseudo-\\differential operator in $L^{2}(\mathbb{R}^{l})$, if $f\in \mathcal{F}_{\rho,\sigma}^{\alpha}$. Let $\hat{f}$ be the Fourier transform of $f$. Since $\|\hat{f}\|_{L^{1}}\leq \|f\|_{\rho,\sigma}^{\alpha}$, we have
\begin{align*}
\|F\|_{L^{2}\rightarrow L^{2}}\leq \int_{\mathbb{R}^{2l}}|\hat{f}(s)|{\rm d}s\equiv \|\hat{f}\|_{L^{1}},\quad \|F\|_{L^{2}\rightarrow L^{2}}\leq \|f\|_{\rho,\sigma}^{\alpha}.
\end{align*}
\end{remark}

We introduce also the space $\mathcal{F}_{\sigma}$ of all functions $f: \mathbb{R}^{2l}\rightarrow \mathbb{C}$ such that
\begin{align*}
\|f\|_{\sigma}^{\alpha}:=\int_{\mathbb{R}^{2l}}|\hat{\tilde{f}}(s)|e^{\sigma|s|^{\frac{1}{\alpha}}}{\rm d}s<\infty.
\end{align*}

\begin{proposition}\label{bk}
In Gevrey function space, if $f,g\in \mathcal{F}_{\rho,\sigma}^{\alpha}$, then, for any $0<\rho'<\rho,0<\sigma'<\sigma$,
\begin{align*}
\|\{f,g\}\|_{\rho',\sigma'}^{\alpha}\leq \frac{1}{(\rho-\rho')(\sigma-\sigma')}\|f\|_{\rho,\sigma}^{\alpha}\|g\|_{\rho,\sigma}^{\alpha}.
\end{align*}
\end{proposition}

\subsection{Main result}
Early efforts in quantum KAM theory focused on the study of quasi-periodic Schr\"{o}dinger operators, such as \cite{MR470318}, and on deriving quantization formulas for fixed values of $h$ (Planck's constant), notably in \cite{MR824989, MR912757}and Popov \cite{MR1770800} constructed quasimodes to approximate eigenvalues and eigenfunctions. These works typically relied on Diophantine conditions on the frequency vectors and provided foundational insights into spectral stability and localization mechanisms.

Later, Popov \cite{MR1770800} advanced the semiclassical analysis of nearly integrable systems by constructing quantum Birkhoff normal forms through abstract unitary conjugation methods, and established a quantum analogue of Arnold’s effective stability theorem via the construction of exponentially accurate quasimodes. This lays the foundation for a systematic treatment of nearly integrable quantum systems. More recently, Charles and San V\~{u} Ng\d{o}c \cite{MR2423760} introduced Birkhoff normal form techniques into the setting of semiclassical pseudodifferential operators, achieving precise spectral asymptotics and uncovering fine structures such as spectral clustering and quantum scarring in the presence of partial resonances.

Consider in $L^{2}(\mathbb{R}^{l})$ the pseudodifferential operator family $\hat{H}(\epsilon,h)=H_{0}(hD_{x})+\epsilon \tilde{P}_{0}(\epsilon,h)$ and assume:
\begin{enumerate}[(A)]
  \item $H=H_{0}(y)+\epsilon P_{0}(x,y)$, where $y\in G\subset \mathbb{R}^{l}$, $x\in \mathbb{T}^{l}$, and $P_{0}$ are Gevrey perturbed functions defined on $G\times \mathbb{T}^{l}$; $H_{0}$ satisfies the nondegenate condition $\det \frac{\partial^{2}H_{0}}{\partial y^{2}}(y)\neq 0$ in $G$. $\epsilon>0$ is a small parameter.
  \item $\hat{H}(\epsilon,h)$ is an elliptic self-adjoint pseudodifferential operator with positive differential order $J$ with the symbol $H$, and the subprincipal symbol of $\hat{H}(\epsilon,h)$ is zero.
  \item $\tilde{P}_{0}\in \Phi_{\rho,\sigma}^{\alpha}$ is a bounded pseudodifferential operator with a holomorphic symbol that decays to zero at infinity; its symbol $P_{0}(x,y)=P_{0}(z)$ is real-valued for $z=(x,y)\in \mathbb{R}^{l}\times\mathbb{R}^{l}$.
  \item There exist $\gamma>0$ such that
  \begin{align}\label{b}
  |\langle k,\omega\rangle|>\frac{\gamma}{\Delta(|k|)},\ k\in\mathbb{Z}^{d}
  \end{align}
  for any $\gamma>0$, and some $\alpha-$approximation function $\Delta$, which is a continuous, strictly increasing, unbounded function $\Delta:[0,\infty)\rightarrow[1,\infty)$ and satisfies
  \begin{align}\label{ab}
  \frac{\log\Delta(t)}{t^{\frac{1}{\alpha}}}\searrow 0,\quad \text {as}\ t\rightarrow\infty,
  \end{align}
  and
  \begin{align}\label{ac}
  \int_{\varsigma}^{\infty}\frac{\log\Delta(t)}{t^{1+\frac{1}{\alpha}}} {\rm d}t<\infty,\quad  \alpha>1
  \end{align}
  with $\varsigma$ being positive constant close to 0. Denote $\Omega_{0}$ the set of all $\omega\in[0,1]^{l}$ fulfilling \eqref{b}. And
  \begin{align*}
  \langle k,\omega\rangle=0,\ k\in \mathbb{Z}^{d_{0}}
  \end{align*}
  for any $y\in O(g,G)$ defined in \eqref{av}.
 \end{enumerate}

\begin{theorem}\label{bc}
Let $(A)-(D)$ be verified; let $h_{0}>0$. There exists $\epsilon^{*}>0$, for all $\epsilon\in [-\epsilon^{*},\epsilon^{*}], \Omega^{\epsilon}\subset\Omega_{0}$ independent of $h$ and $\omega(h,\epsilon)\in\Omega^{\epsilon}$, the spectrum of the perturbed operators family $\hat{H}(\epsilon,h)=H_{0}(hD_{x})+\epsilon \tilde{P}_{0}(\epsilon,h)$ is given by the quantization formula
\begin{align*}
E(n_{y},E_{u},E_{v},\epsilon,h)=\varepsilon(h,\epsilon)+h\sum_{j=1}^{d}\omega_{j}(n_{y}^{j}+\frac{\vartheta_{j}}{4})+\frac{\epsilon}{2}\bigg(\sum_{j=1}^{d_{0}}\lambda_{j}
  (n_{u}^{j}+\frac{1}{2})+\sum_{j=1}^{d_{0}}\tilde{\lambda}_{j}(n_{v}^{j}+\frac{1}{2})\bigg)+O(\epsilon\exp(-ch^{\frac{1}{\alpha-1}})),
\end{align*}
and the quantum state corresponding to energy level $E(n_{y},E_{u},E_{v},\epsilon,h)$ has scarring behavior in stable direction under the semiclassical limit $h\rightarrow0$ on an $d$-dimensional torus.
Here,
\begin{enumerate}[(i)]
  \item $\varepsilon(x,\epsilon):[0,h_{0}]\times[-\epsilon^{*},\epsilon^{*}]\rightarrow \mathbb{R}$ is continuous in $x$ and analytic in $\epsilon$, with $\varepsilon(x,0)=0,\ \varepsilon(0,\epsilon)=0$;
  \item $\omega(x,\epsilon):[0,h_{0}]\times[-\epsilon^{*},\epsilon^{*}]\rightarrow \mathbb{R}$ is continuous in $x$ and analytic in $\epsilon$, with $\omega(x,0)=\omega$;
  \item $\mathcal{R}(y,z,\epsilon):\mathbb{R}^{d}\times\mathbb{R}^{l}\times[-\epsilon^{*},\epsilon^{*}]\rightarrow \mathbb{R}$ is continuous in $(y,z,\epsilon)$;
  \item $|\Omega^{\epsilon}-\Omega_{0}|\rightarrow0$ as $\epsilon\rightarrow0$.
\end{enumerate}
\end{theorem}

We investigate the spectral structure and eigenfunctions of general classes of pseudodifferential operators under both resonant and more general non-resonant conditions, extending beyond the classical Diophantine setting. Working within the Gevrey function space framework, we show that the spectral remainder is exponentially small, rather than of order $O(\epsilon h^{2})$. Moreover, based on our analysis, we establish that the associated eigenfunctions exhibit semiclassical scarring behavior, in a way distinct from the results in \cite{MR2164607}.

In the presence of partial resonances among the classical frequencies, the perturbed quantum Hamiltonian exhibits a characteristic spectral structure of the form:
\begin{align*}
E(n_{y},E_{u},E_{v},\epsilon,h)=\varepsilon(h,\epsilon)+h\sum_{j=1}^{d}\omega_{j}(n_{y}^{j}+\frac{\vartheta_{j}}{4})+\frac{\epsilon}{2}\bigg(\sum_{j=1}^{d_{0}}\lambda_{j}
  (n_{u}^{j}+\frac{1}{2})+\sum_{j=1}^{d_{0}}\tilde{\lambda}_{j}(n_{v}^{j}+\frac{1}{2})\bigg)+O(\epsilon\exp(-ch^{\frac{1}{\alpha-1}})).
\end{align*}
This spectral pattern reveals a superposition of a quasi-periodic structure in the non-resonant directions and a fine clustering of energy levels in the resonant ones. The quadratic terms in the resonant directions—surviving the unitary normal form reduction—give rise to this clustering effect, laying the groundwork for the emergence of quantum scarring, where eigenstates exhibit localization near lower-dimensional invariant tori. In particular, this structure enables the construction of localized quasimodes and aligns naturally with existing scarring theories, such as the general geometric conditions proposed by Gomes \cite{MR4404789} (or \cite{c}), thereby uncovering a deep connection between classical resonance geometry and the localization phenomena of quantum states.

\subsection{The outline of this paper}
In Section 2, we reduce a perturbed Hamiltonian function $H$ to the form of the normal form \eqref{bp} under the frequency conditions $(D)$. This process is similar to \cite{MR1025685}. In Section 3, we use KAM iteration to obtain the final normal form \eqref{au}. It is worth noting that this normal form increases with parameters $\epsilon$ one by one, which provides the basis for the proof in Section 5. In Section 4, we generalize the classical results in Section 2 to the quantum case, and also use the iterative procedure to obtain the quantum normal form of the perturbed quasi-differential operator $\hat{H}(\epsilon,h)$, i.e., the form \eqref{ae}. In Section 5, we also mainly give the estimation of the frequency set. In Section 6, we give the scarring behavior of quantum states on a $d$-dimensional torus in the semiclassical limit $h\rightarrow0$, with the help of the quasi-mode support of a pseudodifferential operator on the $d$-dimensional torus.

\section{Reduction of Normal form in the classical sense}
In this section, we consider a Gevrey Hamiltonian
\begin{align}\label{a}
H(x,y,\epsilon)=H_{0}(y)+\epsilon P_{0}(x,y).
\end{align}
$H_{0}$ is $g$-nondegenerate, that is, $H_{0}$ is nondegenerate and $\det K'^{T}\frac{\partial^{2}H_{0}}{\partial y^{2}}(y)K'\neq0$ for $y\in O(g,G)$ which is to be specialised below.
In what follows, we reduce the perturbed Hamiltonian of the form \eqref{a} to the following form:
\begin{align}
H_{1}(x,y,z,\epsilon)=\epsilon\mathcal{N}_{1}(0)+\langle\omega_{1},y\rangle+\frac{\epsilon}{2}\langle z,M_{1}z\rangle+\epsilon\mathcal{R}_{1}(y,z,\epsilon)+\epsilon P_{1}(x,y,z,\epsilon)
\end{align}
following from \cite{MR1730569} (or \cite{MR1025685}). For the sake of completeness, we briefly write the proof again.

For any given subgroup $g$,
\begin{align}\label{av}
O(g,G)=\{y\in G: \langle k,\omega(y)\rangle=0, k\in g\}
\end{align}
is a $d=l-d_{0}$ dimensional surface, which is called a $g$-resonant surface. By group theory, there are integer vectors $\tau'_{1},\cdots,\tau'_{d}\in \mathbb{Z}^{l}$ such that $\mathbb{Z}^{l}$ is generated by $\tau_{1},\cdots,\tau_{d_{0}},\tau'_{1},\cdots,\tau'_{d}$ and $\det K_{0}=1$, where $K_{0}=(K_{*},K'), K_{*}=(\tau'_{1},\cdots,\tau'_{d}),K'=(\tau_{1},\cdots,\tau_{d_{0}})$ are $l\times l, l\times d$ and $l\times d_{0}$, respectively.

For the subgroup $g$ of $\mathbb{Z}^{l}$, let
\begin{align*}
h_{0}(\phi,y)=\sum_{k\in g}P_{k}e^{i\langle k,x\rangle}=\sum_{l\in \mathbb{Z}^{d_{0}}}P_{K^{\prime}l}e^{i\langle l,\phi\rangle},
\end{align*}
where $\phi=K^{\prime T}x$. Clearly, $h_{0}$ has at least $d_{0}+1$ critical points on $T^{d_{0}}$. Moreover, there are at least $2^{d_{0}}$ critical points if all of them are nondegenerate (see \cite{MR163331}).

Let $\phi_{0}$ be a nondegenerated critical point of $h_{0}(\phi,y)$, i.e. $\frac{\partial h_{0}}{\partial \phi}(\phi_{0},y)=0$, and  $\frac{\partial^{2}h_{0}}{\partial y^{2}}(\phi,y)$ is nonsingular.

Let
\begin{align}
\Gamma=K_{0}^{T}\frac{\partial^{2}H_{0}}{\partial y^{2}}(y_{0})K_{0}=
\begin{pmatrix}
\Gamma_{11} & \Gamma_{12} \\
\Gamma_{21} & \Gamma_{22}
\end{pmatrix},
\end{align}
where $\Gamma_{11},\Gamma_{12},\Gamma_{21},\Gamma_{22}$ are $d\times d, d\times d_{0},d_{0}\times d, d_{0}\times d_{0}$ matrices, respectively, and $\Gamma_{12}=\Gamma_{21}^{T},\Gamma_{22}=K_{0}^{\prime T}\frac{\partial^{2}H_{0}}{\partial y^{2}}(y_{0})K_{0}^{\prime}(\equiv\hat{\Gamma})$. We expand the Hamiltonian $H(x,y,\epsilon)$ at any $y_{0}\in O(g,G)$, by Taylor's formula, into the following form:
\begin{align*}
H(x,y,\epsilon)=H_{0}(y_{0})+\langle\omega(y_{0}),y-y_{0}\rangle+\frac{1}{2}\langle\frac{\partial^{2}H_{0}}{\partial y^{2}}(y_{0})(y-y_{0}),y-y_{0}\rangle+\epsilon P_{0}(x,y)+O(|y-y_{0}|^{3})
\end{align*}
up to a constant. Then, we perform a symplectic coordinate transformation
\begin{equation}
\Phi_{g}: (\theta,Y)\rightarrow (x,y),
\end{equation}
where $y-y_{0}=K_{0}Y,\theta=K_{0}^{T}x$. Then
\begin{align*}
H(\theta,Y)&=\langle\omega^{*},Y^{\prime}\rangle+\frac{1}{2}\langle Y,\Gamma(y_{0})Y\rangle+\epsilon\bar{P_{0}}(\theta,Y)+O(|Y|^{3})\\
&=\langle\omega^{*},Y^{\prime}\rangle+\frac{1}{2}\langle Y^{\prime\prime},\Gamma_{22}Y^{\prime\prime}\rangle+\epsilon\bar{P_{0}}(\theta,Y)+O(|Y|^{3})+\frac{1}{2}\langle Y^{\prime},\Gamma_{11}Y^{\prime}\rangle+\frac{1}{2}\langle Y^{\prime\prime},\Gamma_{12}Y^{\prime}\rangle+\frac{1}{2}\langle Y^{\prime},\Gamma_{22}Y^{\prime\prime}\rangle,
\end{align*}
where $\omega^{*}=K_{*}^{T}\omega(y_{0}),Y^{\prime}=(Y_{1},\cdots,Y_{m})^{T},Y^{\prime\prime}=(Y_{m+1},\cdots,Y_{l})^{T}$,
\begin{align*}
\bar{P}(\theta,Y)=P((K^{T}_{0})^{-1}\theta,y_{0}+K_{0}Y).
\end{align*}
We can find canonical transformations $\Psi_{\epsilon}$ with generating functions S: $I\rightarrow Y=(Y',Y''), \theta\rightarrow\phi=(\phi',\phi'')$
\begin{align*}
S=\langle I,\phi\rangle+\epsilon \sum_{k\in \mathbb{Z}^{d}\setminus \{0\}}\frac{\sqrt{-1}h_{k}}{\langle\omega,k\rangle}(\phi'',\omega)e^{i\langle k,\phi'\rangle}
\end{align*}
with $h_{k}=\int_{0}^{2\pi}\bar{P}(\phi,0)e^{i\langle k,\phi'\rangle}{\rm d} \phi'$. Then
\begin{align}
& Y'=I'+\sqrt{-1}\epsilon\sum_{k\in \mathbb{Z}^{d}\setminus \{0\}}k\frac{\sqrt{-1}h_{k}}{\langle\omega,k\rangle}e^{i\langle k,\phi'\rangle},\\
& Y''=I''+O(\epsilon), \quad \theta=\phi.
\end{align}
For any given frequency $\omega\in \Omega(g,G)$, the following condition is satisfied
\begin{align*}
|\langle k,\omega\rangle|>\frac{\gamma}{\Delta(|k|)},\quad k\in \mathbb{Z}^{d}\setminus\{0\}
\end{align*}
for any $\gamma>0$, and some $\alpha-$approximation function $\Delta$, which is a continuous, strictly increasing, unbounded function $\Delta:[0,\infty)\rightarrow[1,\infty)$ and satisfies \eqref{ab},\eqref{ac}.  From \eqref{b} it follows that $S$ is a Gevrey function.  Under this transformation, the new Hamiltonian function reads as
\begin{align}\label{ad}
H(\phi,I)&=\langle\omega,I'\rangle+\frac{1}{2}\langle I'',\Gamma_{22}(\omega)I''\rangle+\epsilon h_{0}(\phi'',\omega)+\frac{1}{2}\langle I',\Gamma_{11}(\omega)I'\rangle+\frac{1}{2}\langle I',\Gamma_{12}(\omega)I''\rangle\\ \notag
&~~+\frac{1}{2}\langle I'',\Gamma_{21}(\omega)I'\rangle+O(\epsilon I)+O(|I|^{3}).
\end{align}
We have assumed that that $h_{0}(\phi'',\omega)$ has a nondegenerate critical point, denoted by $\phi''_{0}$. Without loss of generality, we assume $\phi''_{0}=0$ up to a linear coordinate transformation. \eqref{ad} is equivalent to the following:
\begin{align}
H(\phi,I)&=\langle\omega,I'\rangle+\frac{1}{2}\langle I'',\Gamma_{22}(\omega)I''\rangle+\frac{1}{2}\epsilon\langle\frac{\partial^{2}h_{0}}{\partial \phi^{2}}(0,\omega)\phi'',\phi'' \rangle+\frac{1}{2}\langle I',\Gamma_{11}(\omega)I'\rangle+\frac{1}{2}\langle I',\Gamma_{12}(\omega)I''\rangle\\ \notag
&~~+\frac{1}{2}\langle I'',\Gamma_{21}(\omega)I'\rangle+O(\epsilon I)+O(|I|^{3})+\epsilon O(|\phi''|^{3}).
\end{align}
We also take $I=\epsilon^{1/3}\bar{I}$ to reduce some less significant terms to a new perturbation and let
\[
z=(\bar{I}'',\phi''), M=
\begin{pmatrix}
U_{0}(\omega) & 0 \\
0 & V_{0}(\omega)
\end{pmatrix}.
\]
We replace $\bar{I}',\phi',\bar{I}'',\phi'', \epsilon^{\frac{1}{2}},\Gamma_{22}$ and $\frac{\partial^{2} h_{0}}{\partial \phi^{2}}(0,\omega)$ by $x,y,u,v,U_{0}(\omega)$ and $V_{0}(\omega)$, respectively, it arrives that
\begin{align*}
H(x,y,u,v,\epsilon)=\langle\omega,y\rangle+\frac{\epsilon}{2}\langle z,M(\omega)z\rangle+\epsilon P_{1}(x,y,u,v,\epsilon).
\end{align*}
Set
\begin{align}\label{at}
&\varepsilon_{1}=\mathcal{N}_{1}(0),\quad \omega_{1}(\epsilon)=\omega+\epsilon(\nabla_{y}\mathcal{N}_{1})(0),\quad M_{1}(\omega)=M(\omega)+\epsilon(\nabla_{z}^{2}\mathcal{N}_{1})(0),\\ \notag
&\mathcal{R}_{1}(y,z,\epsilon)=\mathcal{N}_{1}(y)-\mathcal{N}_{1}(0)-\langle (\nabla_{y}\mathcal{N}_{1})(0),y\rangle-\frac{\epsilon}{2}\langle z,(\nabla_{z}^{2}\mathcal{N}_{1})(0)z\rangle.
\end{align}
Finally, we get
\begin{align}\label{bp}
H_{1}(x,y,z,\epsilon)&=\epsilon\mathcal{N}_{1}(0)+\langle\omega_{1},y\rangle+\frac{\epsilon}{2}\langle z,M_{1}(\omega)z\rangle+\epsilon\mathcal{R}_{1}(y,z,\epsilon)+\epsilon P_{1}(x,y,z,\epsilon)\\ \notag
&=N+\epsilon\mathcal{R}_{1}(y,z,\epsilon)+\epsilon P_{1}(x,y,z,\epsilon),
\end{align}
where $z=(u,v)$ represents the variable in the resonance direction.
\section{The KAM Step}
The KAM iteration process consists of infinitely many KAM steps. From each cycle of KAM steps, one can find the constructions and estimates of the desired symplectic
changes and their domains, perturbed frequencies, and new perturbations.

We use $\Psi_{\epsilon}$ to denote the canonical flow generated by the Hamiltonian $F$ at time $\epsilon$. We define two families of transformations of the form:
\begin{align}\label{ai}
& \chi_{\epsilon}^{p}=\Psi_{\epsilon}^{1}\circ\Psi_{\epsilon}^{2}\cdots\circ\Psi_{\epsilon}^{p},\quad p=1,2,\cdots,\\
& \zeta_{\epsilon}^{s}=\Psi_{\epsilon}^{p}\circ\Psi_{\epsilon}^{p-1}\cdots\circ\Psi_{\epsilon}^{s}
\end{align}
such that after $p$ times iterations, we have the following
\begin{align}\label{au}
H_{p}=\varepsilon_{p}(\epsilon)+\langle\omega_{p},y\rangle+\frac{\epsilon}{2}\langle z,M_{p}(\omega)z\rangle+\epsilon\sum_{s=2}^{p}\Psi_{\epsilon}^{s}\circ\mathcal{R}_{s-1}\epsilon^{s-2}+\epsilon^{p}\mathcal{R}_{p}+\epsilon^{p} P_{p}(x,y,z,\epsilon),
\end{align}
where
\begin{align*}
& \varepsilon_{p}(\epsilon)=\sum_{s=1}^{p}\mathcal{N}_{s}(0)\epsilon^{s},\\
& \omega_{p}=\omega+\sum_{s=1}^{p}\omega_{s}\epsilon^{s},\ \omega_{s}=\nabla_{y}\mathcal{N}_{s}(0),\\
& M_{p}=M+\sum_{s=1}^{p}\nabla_{z}^{2}\mathcal{N}_{s}(0)\epsilon^{s}.
\end{align*}

\subsection{Linearized equation and the perturbation estimation}
We will look for $F\in\mathcal{F}_{\rho,\sigma}^{\alpha}$ such that
\begin{align}\label{an}
\{N,F\}+\epsilon\tilde{R}+\epsilon\langle P_{001},z\rangle=0,
\end{align}
where
\begin{align}\label{ao}
F=\sum_{0\neq|k|\leq K_{+}}(F_{k00}+\langle F_{k10},y\rangle+\langle F_{k01},z\rangle+\langle z,F_{k02}z\rangle)e^{i\langle k,x\rangle}+\langle F_{001},z\rangle.
\end{align}
Let $R$ be a cutoff of $P_{1}$ which depends on $\epsilon$, we assume that $R$ can be expanded as follows:
\begin{align}\label{as}
R=\sum_{|k|\leq K_{+}}(\epsilon P_{k00}+\langle\epsilon P_{k10},y\rangle+\langle P_{k01},z\rangle+\langle z,\epsilon P_{k02}z\rangle)e^{i\langle k,x\rangle}.
\end{align}
Then
\begin{align}\label{bo}
H_{1}\circ\Psi_{\epsilon}&=(N+\epsilon R)\circ\Psi_{\epsilon}+\epsilon(P_{1}-R)\circ\Psi_{\epsilon}+\epsilon\mathcal{R}_{1}\circ\Psi_{\epsilon}\\
&=N+\epsilon[R]-\epsilon\langle P_{001},z\rangle+\epsilon\int_{0}^{1}\{R_{t},F\}\circ\Psi_{\epsilon}{\rm d}t+\epsilon(P_{1}-R)\circ\Psi_{\epsilon}+\epsilon\mathcal{R}_{1}\circ\Psi_{\epsilon}\\
&=N_{+}+P_{+}+\epsilon\mathcal{R}_{1}\circ\Psi_{\epsilon},
\end{align}
where
\begin{align*}
&[R]=\int_{T^{d}}R(x,\cdot){\rm d}x,\\
&\tilde{R}=R-[R],\ R_{t}=(1-t)\epsilon([R]-R-\langle P_{001},z\rangle)+\epsilon R,\\
&N_{+}=N+\epsilon[R]-\epsilon\langle P_{001},z\rangle,\ P_{+}=\int_{0}^{1}\{R_{t},F\}\circ\Psi_{\epsilon}{\rm d}t+\epsilon(P_{1}-R)\circ\Psi_{\epsilon}.
\end{align*}
In what follows, we will solve equation \eqref{an}: Substituting equations \eqref{ao} and \eqref{as} into \eqref{an}, and comparing coefficients, we have
\begin{align*}
&i\langle k,\omega\rangle F_{k00}=\epsilon^{2} P_{k00},\\
&i\langle k,\omega\rangle F_{k10}=\epsilon^{2} P_{k10},\\
&\bigg[-i\langle k,\omega\rangle I_{2d_{0}}+ MJ\bigg]F_{k01}=-\epsilon P_{k01},\\
&\bigg[-i\langle k,\omega\rangle I_{4d_{0}^{2}}+ MJ\otimes I_{2d_{0}}+ I_{2d_{0}}\otimes(MJ)\bigg]F_{k02}=-\epsilon^{2} P_{k02},\\
& MF_{001}=-\epsilon P_{001}.
\end{align*}
To control the norm of $F$, we solve them on the set
\begin{align*}
\Omega_{+}=\bigg\{\omega\in\Omega:|\langle k,\omega\rangle|\geq\frac{\gamma}{\Delta(|k|)},\ |\det A_{1}|>\frac{\gamma^{2d_{0}}}{\Delta^{2d_{0}}(|k|)},\ |\det A_{2}|>\frac{\gamma^{4d^{2}_{0}}}{\Delta^{4d^{2}_{0}}(|k|)},\ \text{for}\ k\in \mathbb{Z}^{d}, 0<|k|\leq K_{+} \bigg\},
\end{align*}
where
\begin{align*}
&A_{1}=-i\langle k,\omega\rangle I_{2d_{0}}+ MJ,\\
&A_{2}=-i\langle k,\omega\rangle I_{4d_{0}^{2}}+ MJ\otimes I_{2d_{0}}+ I_{2d_{0}}\otimes(MJ).
\end{align*}
The solution to the above equations are estimated as follows:
\begin{align*}
&\|F_{k00}\|_{\sigma}^{\alpha}\leq C\epsilon^{2}\|P_{k00}\|_{\sigma}^{\alpha}\gamma^{-1}\Delta(|k|),\\
&\|F_{k10}\|_{\sigma}^{\alpha}\leq C\epsilon^{2}\|P_{k10}\|_{\sigma}^{\alpha}\gamma^{-1}\Delta(|k|),\\
&\|F_{k01}\|_{\sigma}^{\alpha}\leq C\epsilon\|P_{k01}\|_{\sigma}^{\alpha}\gamma^{-2d_{0}}\Delta^{2d_{0}}(|k|),\\
&\|F_{k02}\|_{\sigma}^{\alpha}\leq C\epsilon^{2}\|P_{k02}\|_{\sigma}^{\alpha}\gamma^{-4d^{2}_{0}}\Delta^{4d^{2}_{0}}(|k|),\\
&\|F_{001}\|_{\sigma}^{\alpha}\leq C\epsilon\|P_{001}\|_{\sigma}^{\alpha}
\end{align*}
in the sense of Gevrey function, where $\|F_{k02}\|$ is defined to be the maximum of the norm of the entries. By formula \eqref{ao}, the norm of $F$ is
\begin{align*}
\|F\|_{\rho-r,\sigma}^{\alpha}\leq C\sigma^{2}\|P\|_{\rho,\sigma}^{\alpha}e^{4d_{0}^{2}a T^{\frac{1}{\alpha}}}
\end{align*}
with $\eta=r$ and $T>0$ in Lemma \ref{ba}. Here, the constant $C$ may depend on $\epsilon$.
We mainly focus on the regularity in the $x$ direction. Assuming that the other Fourier coefficients satisfy certain estimates, these homology equations should be solvable.
Let $P_{kjq}$ be the Fourier coefficient of $P$, then we have
\begin{align*}
\|P_{kjq}\|_{\sigma}^{\alpha}\leq C\|P\|_{\rho,\sigma}^{\alpha}e^{-\rho|k|^{\frac{1}{\alpha}}},
\end{align*}
where $\alpha$ represents the Gevrey index of the function $P_{1}$ with respect to $x$ and $\rho>0$.

Next, we complete the last goal of this subsection.  We estimate the new perturbation term in the following:
\begin{align*}
P_{+}=\int_{0}^{1}\{R_{t},F\}\circ\Psi_{\epsilon}{\rm d}t+\epsilon(P_{1}-R)\circ\Psi_{\epsilon}.
\end{align*}

We estimate the norm of $\|P_{+}\|_{\rho-r,\sigma-s}^{\alpha}$ in the spirit of \cite{MR1730569} by dividing it into two parts:
Using Proposition \ref{bk}, we obtain
\begin{align*}
\bigg\|\int_{0}^{1}\{R_{t},F\}\circ\Psi_{\epsilon}{\rm d}t\bigg\|_{\rho-r,\sigma-s}^{\alpha}\leq \frac{C}{rs}\|R\|_{\rho,\sigma}^{\alpha}\|F\|_{\rho,\sigma}^{\alpha}\leq \frac{C}{rs}\|P\|_{\rho,\sigma}^{\alpha}\|F\|_{\rho,\sigma}^{\alpha}.
\end{align*}
Since $R$ is the truncation of $P_{1}$, we have
\begin{align*}
P_{1}-R&=\bigg(\sum_{|k|\geq K_{+}}+\sum_{|k|\leq K_{+},2|j|+|k|\geq3}\bigg)P_{kjq}e^{i\langle k,x\rangle}y^{j}z^{q}\\
&=I+II.
\end{align*}
We estimate $P_{1}-R$ in the space $\mathcal{F}_{\rho-r,\sigma-s}^{\alpha}$. First,
\begin{align*}
\|I\|_{\rho-r,\sigma-s}^{\alpha}\leq \sum_{|k|\geq K_{+}}\|P\|_{\rho,\sigma}^{\alpha}e^{-(\rho-(\rho-r))|k|^\frac{1}{\alpha}}\leq \|P\|_{\rho,\sigma}^{\alpha}\int_{K_{+}}^{\infty} \lambda^{d}e^{-\lambda^{\frac{1}{\alpha}}r}{\rm d}\lambda\leq C\|P\|_{\rho,\sigma}^{\alpha}
\end{align*}
with
\begin{align*}
\int_{K}^{\infty}\lambda^{d}e^{-\lambda^{\frac{1}{\alpha}}r}{\rm d}\lambda\leq C.
\end{align*}
In the same manner,
\begin{align*}
\|II\|_{\rho,\sigma}^{\alpha}\leq C\|P\|_{\rho,\sigma}^{\alpha},
\end{align*}
where $C$ is a constant which depends on $K_{+},\alpha,\sigma-s,\epsilon$.
From this, we deduce
\begin{align*}
\epsilon\|(P-R)\circ\Psi_{\epsilon}\|_{\rho-r,\sigma-s}^{\alpha}\leq C\|P\|_{\rho,\sigma}^{\alpha}.
\end{align*}
We conclude that
\begin{align*}
\|P_{+}\|_{\rho-r,\sigma-s}^{\alpha}\leq C\sigma^{2}\frac{1}{rs}\big(\|P\|_{\rho,\sigma}^{\alpha}\big)^{2}e^{4d_{0}^{2}a T^{\frac{1}{\alpha}}}+C\|P\|_{\rho,\sigma}^{\alpha},
\end{align*}
and $\|P_{+}\|_{\rho-r,\sigma-s}^{\alpha}$ is of at least order $O(\epsilon^{2})$.
Set
\begin{align*}
&\varepsilon_{2}=\varepsilon_{1}+\epsilon^{2}\mathcal{N}_{2}(0),\quad \omega_{2}(\epsilon)=\omega_{1}+\epsilon^{2}(\nabla_{y}\mathcal{N}_{2})(0),\quad M_{2}(\omega)=M_{1}(\omega)+\epsilon^{2}(\nabla_{z}^{2}\mathcal{N}_{1})(0),\\ \notag
&\mathcal{R}_{2}(y,z,\epsilon)=\mathcal{N}_{2}(y)-\mathcal{N}_{2}(0)-\langle (\nabla_{y}\mathcal{N}_{2})(0),y\rangle-\frac{\epsilon}{2}\langle z,(\nabla_{z}^{2}\mathcal{N}_{2})(0)z\rangle.
\end{align*}
Finally, we get
\begin{align*}
H_{2}(x,y,z,\epsilon)&=\varepsilon_{2}+\langle\omega_{2},y\rangle+\frac{\epsilon}{2}\langle z,M_{2}(\omega)z\rangle+\epsilon\mathcal{R}_{1}(y,z,\epsilon)\circ\Psi_{\epsilon}+\epsilon^{2} \mathcal{R}_{2}(y,z,\epsilon)+\epsilon^{2} P_{1}(x,y,z,\epsilon).
\end{align*}

\subsection{Iteration}
First, we set the parameters. Let $\epsilon,K,\gamma,\rho,\sigma$ and $\alpha$ keep fixed. Define, for $p\geq1$:
\begin{align*}
&\sigma_{p}:=\frac{\sigma}{4p^{2}},\quad s_{p}:=s_{p-1}-\sigma_{p},\quad \rho_{p}:=\frac{\rho}{4p^{2}},\quad r_{p}:=r_{p-1}-\rho_{p},\\
&\gamma_{p}:=\gamma_{\infty}=\gamma_{0}-C\gamma_{1},\quad K_{p}=pK,\quad \epsilon_{p}=\epsilon^{p}\|q_{p-1}\|_{r_{p-1},s_{p-1}}^{\alpha}.
\end{align*}
The initial values of the parameter sequences are chosen as follows:
\begin{align*}
\gamma_{0}:=\gamma;\ s\neq s_{0}:=\sigma,\ r\neq r_{0}:=\rho;\ \epsilon_{0}:=0.
\end{align*}
We always have
\begin{align*}
H_{p}(x,y,z,\epsilon)&=\varepsilon_{p}+\langle\omega_{p},y\rangle+\frac{\epsilon}{2}\langle z,M_{p}z\rangle+\epsilon\sum_{s=2}^{p}\Psi_{\epsilon}^{s}\circ\mathcal{R}_{s-1}\epsilon^{s-2}+\epsilon^{p}P_{p}(x,y,z,\epsilon)\\
&=N_{p}+\epsilon\sum_{s=2}^{p}\Psi_{\epsilon}^{s}\circ\mathcal{R}_{s-1}\epsilon^{s-2}+\epsilon^{p}P_{p}.
\end{align*}
We set
\begin{align*}
& \varepsilon_{p}(\epsilon)=\sum_{s=1}^{p}\mathcal{N}_{s}(0)\epsilon^{s},\\
& \omega_{p}=\omega+\sum_{s=1}^{p}\omega_{s}\epsilon^{s},\ \omega_{s}=\nabla_{y}\mathcal{N}_{s}(0),\\
& M_{p}=M+\sum_{s=1}^{p}\nabla_{z}^{2}\mathcal{N}_{s}(0)\epsilon^{s},
\end{align*}
and
\begin{align*}
\mathcal{R}_{p}(y,z,\epsilon)=\mathcal{N}_{p}(y)-\mathcal{N}_{p}(0)-\langle (\nabla_{y}\mathcal{N}_{p})(0),y\rangle-\frac{\epsilon}{2}\langle z,(\nabla_{z}^{2}\mathcal{N}_{p})(0)z\rangle.
\end{align*}
Then, we arrive that
\begin{align*}
H_{p}=\varepsilon_{p}(\epsilon)+\langle\omega_{p},y\rangle+\frac{\epsilon}{2}\langle z,M_{p}(\omega)z\rangle+\epsilon\sum_{s=2}^{p}\Psi_{\epsilon}^{s}\circ\mathcal{R}_{s-1}\epsilon^{s-2}+\epsilon^{p}\mathcal{R}_{p}(y,z,\epsilon)+\epsilon^{p} P_{p}(x,y,z,\epsilon).
\end{align*}

By iteration, we can get
\begin{align*}
\|P_{p}\|_{r_{p},s_{p}}^{\alpha}\leq C\bigg(\sigma^{2}\bigg(\frac{4p^{2}}{\rho}\bigg)^{p}\bigg(\frac{4p^{2}}{\sigma}\bigg)^{p}\big(\|P\|_{r,s}^{\alpha}\big)^{2p}e^{4d_{0}^{2}a T^{\frac{1}{\alpha}}}+\big(\|P\|_{r,s}^{\alpha}\big)^{p}\bigg).
\end{align*}

\section{Quantum}
We then generalize this result to quantum. First, we quantize the symplectic maps $\Phi_{g},\Psi_{\epsilon}$. Following from Chap11.2 in \cite{MR2952218}, we can see that the symplectic map can be quantized, that is, we can get a unitary quasidifferential operator $F_{1h}$, such that
\begin{align*}
F_{1h}^{-1}\circ \hat{H}(\epsilon,h)\circ F_{1h}
\end{align*}
is an $h$-quasi-differential operator.
We can write $F_{1h}=A_{h}+h B_{h}$ by Popov \cite{MR1770800}, where the principal symbol of $A_{h}$ is equal to 1, and then we solve a linear equation for the real part of the principal symbol of $B_{h}$. Similarly, we can construct another $h$-quasidifferential operator $F_{2h}$, whose  distribution kernel has the following form
\begin{align*}
(2\pi h)^{-n}\int e^{{\rm i} S(I,\phi)/h} b(x,I;h){\rm d}I.
\end{align*}
Here, the function $S(I,\phi)$ is the generating function of the canonical transformation $\Psi_{\epsilon}$. Set $F_{h}(\epsilon)=F_{1h}(\epsilon)\circ F_{2h}(\epsilon)$, we have:
\begin{lemma}
The operator $F^{*}_{h}\circ\hat{H}(\epsilon,h)\circ F_{h}$ is an h-PDO with a symbol in the function space $\mathcal{F}_{\rho,\sigma}^{\alpha}$. Moreover, the principal symbol of $F^{*}_{h}\circ\hat{H}(\epsilon,h)\circ F_{h}$ equals to $(H\circ\Phi_{g})\circ\Psi_{\epsilon}$ and the subprincipal symbol is 0.
\end{lemma}
Next, we construct the unitary operator $U_{h}$. Suppose that $A_{h}(\epsilon)$ is an elliptic pseudodifferential operator with a symbol $a$, and the principal symbol $a_{0}=1$. Set $W_{h}(\epsilon)=F_{h}(\epsilon)\circ A_{h}(\epsilon)$, where $F_{h}(\epsilon)=F_{1h}(\epsilon)\circ F_{2h}(\epsilon)$. Using this construction, we obtain the relation:
\begin{align*}
\hat{H}(\epsilon,h)\circ V_{h}(\epsilon)=V_{h}(\epsilon)\circ(P_{h}^{0}(\epsilon)+R_{h}(\epsilon)),
\end{align*}
where $P_{h}^{0}(\epsilon)$ and $R_{h}(\epsilon)$ have the desired properties. However, $V_{h}(\epsilon)$ may not be a unitary operator. To resolve this, we set $\tilde{Q}_{h}(\epsilon)=(V_{h}^{*}(\epsilon)\circ V_{h}(\epsilon))^{-1/2}$ with a symbol $q(\varphi,I;h,\epsilon)$. Now, by setting $U_{h}(\epsilon)=V_{h}(\epsilon)\circ \tilde{Q}_{h}(\epsilon)$, we construct the desired unitary operator. This operator $U_{h}(\epsilon)$ fulfills the necessary properties for our analysis.

Through the above construction, we can get a unitary operator $U_{h}$ such that
\begin{align*}
U_{h}\circ \hat{H}(\epsilon,h)\circ U_{h}^{-1}&=\varepsilon_{1}(h;\epsilon)I_{d}+P_{0}(h,\omega_{1}(\epsilon))+\frac{\epsilon}{2}\hat{M_{1}}(\omega_{1};\epsilon)+\epsilon R_{1}(h;\epsilon)+\epsilon \tilde{P}_{1}(\epsilon,h)\\
&\triangleq P^{0}=\mathcal{K}^{0}+ \epsilon \tilde{P}_{1}(\epsilon,h).
\end{align*}
Clearly, $\mathcal{K}^{0}$ is an $h-$differential operator whose Weyl symbol $K^{0}$ admits a full asymptotic expansion in $h$. Suppose that $a, p$, and $p^{0}$ denote the symbols of operators $A_{h}, \hat{H}$, and $P^{0}$  respectively.
\begin{theorem}\label{bm}
There exist symbols $a$ and $p^{0}$ given by
\begin{align*}
a\sim \sum_{j=0}^{\infty}a_{j}h^{j},\quad p^{0}\sim \sum_{j=0}^{\infty}p_{j}^{0}h^{j},
\end{align*}
with $a_{0}=1,p_{0}^{0}=K_{0},p_{1}^{0}=0,$ so that
\begin{align*}
p\circ a-a\circ p^{0}\sim 0,
\end{align*}
where $K_{0}$ is the principal symbol of $\mathcal{K}^{0}$.
\end{theorem}
The composite operation of the quasi-differential operator is transformed into a symbolic operation. We only need to prove that Theorem \ref{bm} holds. We can solve the homological equations under non-resonance condition \eqref{b}. The specific process can be found in reference \cite{a}.

For the remainder, we get the error term $\tilde{P}_{1}$ to be exponentially small. This part of the expressions do not explicitly contain the variable $z$ and $x$. On the one hand, the quantization of the classical action variable $y$ is $I_{n}=h(n+\vartheta/4),n\in \mathbb{Z}^{l}$ and $\vartheta$ is the Maslov class of the invariant tori.
On the other hand, for a given set of action variables $E_{\gamma}(\epsilon)=\omega^{-1}(\Omega_{0},\epsilon)$, we expand $\tilde{P}(I_{m},\cdot,\epsilon,h)$ in Taylor series at some $I_{0}\in E_{\gamma}(\epsilon)$, which is a $d$-dimensional variable such that $|I_{0}-I_{m}|=|E_{\gamma}(\epsilon)-I_{m}|=\inf_{I^{\prime}\in E_{\gamma}(\epsilon)}|I^{\prime}-I_{m}|$,  where $|E_{\gamma}(\epsilon)-I_{m}|=\inf_{I^{\prime}\in E_{\gamma}(\epsilon)}|I^{\prime}-I_{m}|$ is the distance to the compact set $E_{\gamma}(\epsilon)$. We get for any $n\in \mathbb{Z}_{+}$,
\begin{align}\label{ay}
|\tilde{P}_{1}(I,\cdot,\epsilon,h)|\leq C^{n+1}n!^{\alpha-1}|I_{0}-I_{m}|^{n}
\end{align}
for all $(I,\cdot,\epsilon,h)\in D\times \mathbb{T}^{d},\ I_{m} \notin E_{\gamma}(\epsilon),\ h\in(0,h_{0}],\ \epsilon\in(0,\epsilon^{*}]$. Using Stirling's formula, we may minimize the  right-hand side with respect to $n\in \mathbb{Z}_{+}$. An optimal choice for $n$ will be
\begin{align}\label{ax}
n\sim (-C^{-1}|I_{0}-I_{m}|)^{-\frac{1}{\alpha-1}},
\end{align}
such that
\begin{align*}
|\tilde{P}_{1}(I,\cdot,\epsilon,h)|\leq C\exp(-C^{-1}|I_{0}-I_{m}|^{-\frac{1}{\alpha-1}})
\end{align*}
for every $\alpha,\ \beta\in\mathbb{Z}_{+}^{d}$ and $(I,\cdot,\epsilon,h)\in D\times\mathbb{T}^{d}\times(0,\epsilon^{*}]\times(0,h_{0}],I_{m}\notin E_{\gamma}(\epsilon)$.
From the definition of the index set $\mathcal{M}_{h}(\epsilon)$ given in equation \eqref{bb}, we can derive that
\begin{align*}
|\tilde{P}_{1}(I,\cdot,\epsilon,h)|\leq C\exp(-ch^{-\frac{1}{\alpha-1}}),
\end{align*}
where the constant $c$ depends on $C^{-1}, L,\alpha$.
Here, we explain why the above form \eqref{ax} is the optimal choice for $n$. The terms on the right side of inequality \eqref{ay} that involves $n$ are as follows:
\begin{align*}
f(n)\triangleq C^{n}n!^{\alpha-1}|I_{0}-I_{m}|^{n}.
\end{align*}
Taking the logarithm of both sides, we have
\begin{align*}
\log f(n)= n\log C+(\alpha-1)\log n!+n\log|I_{0}-I_{m}|.
\end{align*}
Using the logarithmic form of Stirling's approximation, that is,
\begin{align*}
\log n!\sim n\log n-n,
\end{align*}
one obtains
\begin{align*}
\log f(n)= n(\log C+\log|I_{0}-I_{m}|-\alpha+1)+(\alpha-1) n\log n.
\end{align*}
To minimize $\log f(n)$, take the derivative with respect to $n$ and set the derivative to zero.
Finally, since $\tilde{P}_{1}$ has an exponentially small tail estimate on $I$, and the Fourier transform preserves (in some sense) this exponential decay, we know that the decay of $\hat{\tilde{P}}_{1}(s)$ also satisfies
\begin{align*}
|\hat{\tilde{\tilde{P}}}^{0}_{1}(s)|\leq C\exp(-\rho|k|^{\frac{1}{\alpha}}-\sigma|s|^{\frac{1}{\alpha}}-ch^{\frac{1}{\alpha-1}}),
\end{align*}
hence
\begin{align*}
\|\tilde{P}_{1}\|_{\rho,\sigma}^{\alpha}\leq C\exp(-ch^{\frac{1}{\alpha-1}})
\end{align*}
by \eqref{aw}.
We can also iterate the above process. We have the following result:
\begin{proposition}
Let $\omega\in\Omega_{0}$. There exist $epsilon^{*}>0$ and $\forall p\geq1$, a closes set $\Omega_{p}^{\gamma}\subset\Omega_{0}$ such that, if $|\epsilon|<\epsilon^{*}$ and $\omega_{p}\in\Omega_{p}^{\gamma}$, then the followings hold:
\begin{itemize}
  \item [1.]  We can construct two sequences of unitary transformations $\{X_{p}\},\{Y_{p}\}$ in $L^{2}(\mathbb{R}^{l})$ with the property
  \begin{align}\label{ae}
  X_{p}\hat{H}(\epsilon,h)X_{p}=\varepsilon_{p}(h;\epsilon)I_{d}+P_{0}(\omega_{p}(h;\epsilon))+\frac{\epsilon}{2}\hat{M_{p}}(\omega_{p};\epsilon)
  +\epsilon\sum_{s=2}^{p}Y_{s}R_{s-1}(h)Y_{s}^{-1}\epsilon^{s-2}+\epsilon^{p}R_{p}(h;\epsilon)+\epsilon^{p}\tilde{P}_{p},
  \end{align}
  where
  \begin{align*}
  \frac{\epsilon}{2}\hat{M_{p}}(\omega_{p};\epsilon)=\frac{\epsilon}{2}\sum_{j,k}\bigg((U_{0})_{jk}(\omega)Op_{h}^{w}(u_{j}u_{k})+(V_{0})_{jk}(\omega)Op_{h}^{w}(v_{j}v_{k})\bigg);
  \end{align*}
  \item [2.] $X_{p}$ and $Y_{p}$ have the following forms:
  \begin{align}\label{af}
  X_{p}=U_{p}\circ U_{p-1}\cdots\circ U_{1},\\
  Y_{s}=U_{p}\circ U_{p-1}\cdots\circ U_{s};
  \end{align}
  \item [3.] The remainder satisfies the estimate:
 \begin{align}\label{bl}
 \|\tilde{P}_{p}\|_{\rho,\sigma}^{\alpha}\leq C\exp(-cph^{\frac{1}{\alpha-1}}).
 \end{align}
\end{itemize}
\end{proposition}
\begin{remark}
If $U_{0}(\omega)$ and $V_{0}(\omega)$ are diagonalizable (i.e. the resonance variable are completely separable), then the above equation can be further simplified to:
\begin{align*}
\frac{\epsilon}{2}\hat{M_{p}}(\omega_{p};\epsilon)=\frac{\epsilon}{2}\sum_{j=1}^{d_{0}}\bigg(\lambda_{j}Op_{h}^{w}(u_{j}^{2})+\tilde{\lambda}_{j}Op_{h}^{w}(v_{j}^{2})\bigg).
\end{align*}
That is, it is a group of independent resonant direction quantum oscillators.
\end{remark}
\begin{proof}
We proceed by induction. For $p=1$, we already have the corresponding results. To go from step $p-1$ to step $p$, we consider the operator
\begin{align*}
X_{p-1}\hat{H}(\epsilon,h)X_{p-1}^{-1}&:=P_{0}(\omega_{p-1}(h;\epsilon))+\frac{\epsilon}{2}\hat{M_{p}}(\omega_{p-1};\epsilon)+\epsilon\varepsilon_{p-1}(h;\epsilon)I_{d}
+\epsilon\sum_{s=2}^{p-1}Y_{s}R_{s-1}(h)Y_{s}^{-1}\epsilon^{s-2}\\
&~~+\epsilon^{p-1}R_{p-1}(h;\epsilon)+\epsilon^{p-1}\tilde{P}_{p-1}.
\end{align*}
We have to determine and estimate the unitary operator $U_{p}$ transforming it into the form \eqref{ae} via the definitions \eqref{af}. Then, we have at $p^{th}$ iteration step
\begin{align*}
&X_{p}\hat{H}(\epsilon,h)X_{p}^{-1}=U_{p}(X_{p-1}H(\epsilon,h)X_{p-1}^{-1})U_{p}^{-1}=P_{0}(\omega_{p}(h;\epsilon))+\frac{\epsilon}{2}\hat{M_{p}}(\omega_{p};\epsilon)
+\epsilon^{p-1}Q_{p}+\epsilon^{p}\tilde{P}_{p},\\
&Q_{p}=\tilde{P}_{p-1}+[U_{p},P_{0}]/ih,\\
&\tilde{P}_{p}=\epsilon^{-p}\bigg(U_{p}(X_{p-1}H(\epsilon,h)X_{p-1}^{-1})U_{p}^{-1}-P_{0}(\omega_{p}(h;\epsilon))-\frac{\epsilon}{2}\hat{M_{p}}(\omega_{p};\epsilon)
-\epsilon(\tilde{P}_{p-1}+[U_{p},P_{0}]/ih) \bigg).
\end{align*}
Here, $U_{p}$ is a unitary operator constructed as in above abstract conjugation scheme, and the expression for $Q_{p}$ follows from isolating the contribution at order $\epsilon^{p-1}$ in the conjugated operator.
We will look for $U_{p}\in \Phi_{\rho,\sigma}^{\alpha}$ and an operator $N_{p}\in\Phi_{\rho,\sigma}^{\alpha}$ with symbols $u_{p}$ and $\mathcal{N}_{p}$ respectively, such that
\begin{align*}
\{p_{0},u_{p}\}+\mathcal{N}_{p}=q_{p},\ \{p_{0},\mathcal{N}_{p}\}=0.
\end{align*}
We can obtain the existence of $u_{p}\in\mathcal{F}_{\rho,\sigma}^{\alpha},\mathcal{N}_{p}\in \mathcal{F}_{\rho,\sigma}^{\alpha}$ with the stated properties now by directly applying the result for $p=1$. Expanding $\mathcal{N}_{p}$ as the equation \eqref{at} and taking into account the definitions \eqref{af} we immediately check that $U_{p}(X_{p-1}\hat{H}(\epsilon,h)X_{p-1}^{-1})U_{p}^{-1}$ has the form \eqref{ae}.

The norm estimation process for $P_{p}$ is similar to that for $P_{1}$. Next, we only have to prove that the operator $U_{p}R_{p}U_{p}^{-1}$ is an $h$-pseudodifferential operator, assuming by the inductive argument the validity of these properties for $R_{p-1}$. On the other hand, $U_{p}$ is an $h$-pseudodifferential operator of order 0. We can therefore apply the semiclassical Egorov theorem (see e.g. \cite{MR897108}, Chapter 4) to assert that $U_{p}R_{p}U_{p}^{-1}$ is again an $h$-pseudodifferential operator. Denote $\eta(y,z,\epsilon,h)$ the Weyl symbol of $U_{p}R_{p}U_{p}^{-1}$ and it has the expansion
\begin{align*}
\eta(y,z,\epsilon,h)=\sum_{j=0}^{M}h^{j}\eta_{j}(y,z,\epsilon)+O(h^{M+1}).
\end{align*}
\end{proof}
By deriving the resonant quantum normal form for the operator $\hat{H}(\epsilon,h)$, we obtain a clear and structured understanding of the spectral behavior under small perturbations. The normal form reveals the detailed structure of the spectrum, particularly in the context of invariant torus splitting induced by partial frequency resonances. This provides a simplified yet comprehensive picture of how the perturbations influence the quantum system's spectrum. The spectral expression derived from the normal form explicitly depends on quantum numbers, highlighting spectral clustering phenomena and the eigenvalue spacing determined by the quadratic terms in resonant directions. These quadratic terms play a crucial role in the structure of the spectrum, and their presence is directly linked to the occurrence of quantum scarring. This framework also facilitates the analysis of quantum state localization on lower-dimensional invariant tori, which is characteristic of semiclassical quantum systems under resonance conditions.
\begin{proof}[Proof of Theorem \ref{bc}]
At the $p^{th}$ iteration, the frequency is given by
\begin{align}\label{bd}
\omega_{p}(h,\epsilon)=\omega+\sum_{s=1}^{p}\nabla_{y}\mathcal{N}_{s}(h)\epsilon^{s},
\end{align}
and
\begin{align}\label{bg}
M_{p}(\omega_{p},\epsilon)=M+\sum_{s=1}^{p}\nabla_{z}^{2}\mathcal{N}_{s}(h)\epsilon^{s}.
\end{align}
Since $\|\nabla_{z}f(z)\|_{\rho-d,\sigma-s}^{\alpha}\leq \frac{1}{ds}\|f(z)\|_{\rho,\sigma}^{\alpha}$ and $\|\nabla_{z}^{2}f(z)\|_{\rho-d,\sigma-s}^{\alpha}\leq \frac{1}{(ds)^{2}}\|f(z)\|_{\rho,\sigma}^{\alpha}$, by \eqref{bl}, two series \eqref{bd},
\eqref{bg} converges as $p\rightarrow\infty$ for $|\epsilon|<\epsilon^{*}$, if $\epsilon^{*}$ is small enough, uniformly with respect to $h\in [0,h^{*}]$. Let $\omega_{\infty}(h,\epsilon):=\lim_{p\rightarrow\infty}\omega_{p}(h,\epsilon)$ and $M_{\infty}:=\lim_{p\rightarrow\infty}M_{p}(h,\epsilon)$. Then, $\omega(h,\epsilon)$ is in $\Omega_{0}$ with constant $\gamma_{\infty}$. In the same way:
\begin{align*}
\varepsilon(h,\epsilon)=\sum_{s=1}^{\infty}\mathcal{N}_{s}(h)\epsilon^{s},\quad |\epsilon|<\epsilon^{*}.
\end{align*}
We introduce the resonance quantum number, denoted as $n_{res}=(n_{u},n_{v})$.  We can get the spectral structure of the quasi-differential operator $\hat{H}(\epsilon)$ as
\begin{align*}
E(n_{y},E_{u},E_{v},\epsilon,h)=\varepsilon(\epsilon,h)+h\sum_{j=1}^{d}\omega_{j}(n_{y}^{j}+\frac{\vartheta_{j}}{4})+\frac{\epsilon}{2}E_{res}(n_{u},n_{v})+\epsilon \mathcal{R}(h;\epsilon),
\end{align*}
where
\begin{align*}
E_{res}(n_{u},n_{v})=\sum_{j=1}^{d_{0}}\lambda_{j}(n_{u}^{j}+\frac{1}{2})+\sum_{j=1}^{d_{0}}\tilde{\lambda}_{j}
(n_{v}^{j}+\frac{1}{2}),
\end{align*}
here, $\lambda_{j},\tilde{\lambda}_{j}$ is the characteristic frequency in the resonance direction, i.e. $\lambda_{j},\tilde{\lambda}_{j}$ are the eigenvalues of $U_{0}(\omega),V_{0}(\omega)$. Finally, let $\mathcal{R}$ be an asymptotic sum of the power series $\sum_{s=2}^{\infty}Y_{s}R_{s-1}(h)Y_{s}^{-1}\epsilon^{s-1}$.
\end{proof}

\section{Estimate frequency set}
Following from \cite{MR1401420}, we can make an assert that $|\mathcal{T}_{k}(\beta)|\leq C\frac{\beta}{|k|}$. To prove this inequality, first, we set
\begin{align*}
\mathcal{T}_{k}(\beta)=\bigg\{\omega\in[0,1]^{l}: |\langle\omega,k\rangle|\leq\beta,\ |\det A_{1}|\leq\frac{\gamma^{2d_{0}}}{\Delta^{2d_{0}}(|k|)},\ |\det A_{2}|\leq\frac{\gamma^{4d^{2}_{0}}}{\Delta^{4d^{2}_{0}}(|k|)}\bigg\},
\end{align*}
\begin{align*}
\Omega_{1}:=\Omega_{0}-\bigcup_{k\in \mathbb{Z}^{d}\setminus\{0\}}\mathcal{T}_{k}\bigg(\frac{\gamma_{1}}{\Delta(|k|)}\bigg).
\end{align*}
Similarly, we have
\begin{align*}
\bigg|\{\omega:|\langle k,\omega\rangle|\leq \beta\}\bigg|\leq \frac{C}{|k|}\beta.
\end{align*}
We briefly describe the proof.  We introduce the unperturbed frequencies $\zeta=\omega(\xi)$ as parameter over the domain $\Lambda=\omega(\Pi)$ and consider the resonance zones $\mathcal{T}_{k}^{\Lambda}=\omega(\mathcal{T}_{k})$ in $\Lambda$. Now, we consider the frequency set $\{\omega:|\langle k,\omega\rangle|\leq \beta\}$. Let $\phi(\zeta)=\langle k,\omega'(\zeta)\rangle$. Choose a vector $v\in [0,1]^{l}$ such that $\langle k,v\rangle=|k|$ and write $\zeta=rv+w$ with $r\in\mathbb{R},w\in v^{\bot}$. As a function of $r$, we then have, for $t>s$,
\begin{align*}
\langle k,\omega'(\zeta)\rangle|_{s}^{t}=\langle k,\zeta\rangle|_{s}^{t}+\langle k,\omega'(\zeta)=\zeta\rangle|_{s}^{t}\geq \frac{1}{2}|k|(t-s).
\end{align*}
Hence, $\phi(rv+w)|_{s}^{t}\leq\frac{1}{2}|k|(t-s)$ uniformly in $w$. It follows that
\begin{align*}
\{r:rv+w\in\Lambda,|\phi(rv+w)|\leq\delta\}\subset \{r:|r-r_{0}(w)|\leq 2\delta|k|^{-1}\}
\end{align*}
with $r_{0}$ depending miserably on $w$, and hence
\begin{align*}
\bigg|\{\omega:|\langle k,\omega\rangle|\leq \beta\}\bigg|\leq \frac{C}{|k|}\beta.
\end{align*}
From \cite{MR1730569}, we have
\begin{align*}
\bigg|\bigg\{\omega:|g(\langle k,\omega\rangle)|\leq \frac{\gamma_{1}^{4d_{0}^{2}}}{\Delta^{4d^{2}_{0}}(|k|)}\bigg\}\bigg|\leq C\frac{\gamma_{1}}{\Delta(|k|)}.
\end{align*}
Similarly,
\begin{align*}
\bigg|\bigg\{\omega:|\det A_{1,1}|\leq \frac{\gamma_{1}^{2d_{0}}}{\Delta^{2d_{0}}(|k|)}\bigg\}\bigg|\leq C\frac{\gamma_{1}}{\Delta(|k|)}.
\end{align*}
Then, we can arrive that
\begin{align*}
\bigg|\mathcal{T}_{k}(\beta)\bigg|\leq C\frac{\beta}{|k|}.
\end{align*}
Hence, we can write
\begin{align*}
\bigg|\bigcup_{k\in \mathbb{Z}^{d}\setminus\{0\}}\mathcal{T}_{k}\bigg(\frac{\gamma_{1}}{\Delta(|k|)}\bigg)\bigg|\leq \sum_{k\in \mathbb{Z}^{d}\setminus\{0\}}\frac{\gamma_{1}}{\Delta(|k|)|k|}\leq \gamma_{1}\sum_{m=1}^{\infty}\frac{m^{d-1}}{\Delta(m)}<\infty
\end{align*}
through the condition \eqref{ab},\eqref{ac}.
Therefore
\begin{align*}
|\Omega_{0}-\Omega_{1}|\leq O(\gamma_{1}).
\end{align*}

\section{Quantum states}
We can prove that quantum states have scarring phenomena on the $d$-dimensional invariant torus. We need to impose some conditions on this law form to make it linearly independent. In this section, we are mainly concerned with the $m$-dimensional action variable $y$ in the integrable part of the law form about the stable direction, while the resonant direction ($d_{0}$ dimension) can be "encoded" as a secondary term in the law form, and its semiclassical influence is secondary. This variable is quantized and recorded as $I_{m},m\in \mathbb{Z}^{d}$. Except for the parameters $\epsilon$ and $h$, other variables are not explicitly included in the expression of $K^{0}$.

\subsection{Separation of quasieigenvalue}
To analyze spectral concentration in this section, we draw on the idea of using the spectral flow of a one-parameter family of operators, as discussed in Hassell's paper \cite{MR2630052}, a pioneering work on scarring. First, for any the non-degenerate completely integrable Hamiltonian $H_{0}(I)=H(\phi,I;0)$, we consider a one-parameter family of the perturbed Hamiltonian
\begin{align*}
H(\phi,I;\epsilon)\in \mathcal{G}^{\alpha,\alpha,1}(\mathbb{T}^{d}\times D\times(-\epsilon^{*},\epsilon^{*})),\ |\epsilon^{*}|<1.
\end{align*}
We choose $\gamma$ sufficiently small so that the set of non-resonant frequency set $\Omega_{0}$ has positive measure. And without loss of generality, we assume that $D$ is convex by sharking if necessary. We define the index
\begin{align}\label{bb}
\mathcal{M}_{h}(\epsilon)=\{m\in\mathbb{Z}^{d}:|E_{\gamma}(\epsilon)-h(m+\vartheta/4)|\leq Lh\},
\end{align}
where $E_{\gamma}(\epsilon)=\omega^{-1}(\Omega_{0},\epsilon)$, which depends on the parameter $\epsilon$.

According to the previous discussion, we have
\begin{align*}
\mathcal{K}^{0}=\varepsilon_{p}(h;\epsilon)I_{d}+P_{0}(\omega_{p}(h;\epsilon))+\frac{\epsilon}{2}\hat{M_{p}}(\omega_{p};\epsilon)+
\epsilon\sum_{s=2}^{\infty}Y_{s}R_{s-1}(h)Y_{s}^{-1}\epsilon^{s-2}
\end{align*}
with
\begin{align}\label{bn}
\varepsilon(h,\epsilon)=\sum_{s=1}^{\infty}\mathcal{N}_{s}(h)\epsilon^{s},\quad \omega_{\infty}=\omega+\sum_{s=1}^{\infty}\nabla_{y}\mathcal{N}_{s}(h)\epsilon^{s},\quad M_{\infty}=M+\sum_{s=1}^{\infty}\nabla_{z}^{2}\mathcal{N}_{s}(h)\epsilon^{s},
\end{align}
where $P_{0}(h,\omega)$ is the maximal operator in $L^{2}$  generated by the differential expression
\begin{align*}
\tilde{P}_{0}(h,\omega)=h\sum_{i=1}^{d}\omega_{i}z_{i}D_{y_{i}}\equiv h\langle\omega y,\nabla_{y}\rangle.
\end{align*}
$\hat{M}_{p}$ is a pseudodifferential operator by quantizing Weyl symbol $\langle z,M_{p}z\rangle$. Let $K^{0}$ be the symbol of $\mathcal{K}^{0}$, which admits a full asymptotic expansion in $h$, i.e.
\begin{align*}
K^{0}(I,\epsilon,h)=\sum_{j=0}^{\infty}K^{0}_{j}(I,\epsilon)h^{j}
\end{align*}
with the principal symbol
\begin{align*}
K^{0}_{0}=\varepsilon(\epsilon)+\langle\omega_{\infty},y\rangle+\frac{\epsilon}{2}\langle z,M_{\infty}(\omega)z\rangle+\epsilon\sum_{s=2}^{\infty}\Psi_{\epsilon}^{s}\circ\mathcal{R}_{s-1}\epsilon^{s-2}.
\end{align*}
The expressions for $\varepsilon(\epsilon),\omega_{\infty}$ and $M_{\infty}$ in the case $h=0$ coincide with those in Equation \eqref{bn}.
We take quasi-eigenvalue as the form $\lambda(\epsilon;h)=K^{0}(h(m+\vartheta/4);\epsilon,h)$, where $I_{m}=h(m+\vartheta/4),m\in \mathcal{M}_{h}(\epsilon)$.

We make the geometric assumption that the functions $\partial_{\epsilon}K^{0}(I;0,h),\partial^{2}_{\epsilon}K^{0}(I;0,h),\cdots,\partial_{\epsilon}^{d-1}K^{0}(I;0,h)$ are linearly independent. Based on this assumption on the perturbation family $H(\phi,I;\epsilon)$, we assert that the collection $K^{0}, \partial_{\epsilon} K^{0},\cdots,\partial_{\epsilon}^{(d-1)}K^{0}$ forms a local coordinate system on the action domain $D$ ffor all sufficiently small $\epsilon<\epsilon^{*}$ and $h<h_{0}$. This follows from the fact that in dimension $d$, the level sets of these $d$ functions intersect transversally, thereby furnishing a local coordinate chart for the action space.  More precisely, we have the following result:
\begin{proposition}\label{am}
There exist $\epsilon^{*},h_{0}>0$ such that for all $\epsilon<\epsilon^{*}$ and $h<h_{0}$,
\begin{align*}
\eta:I\rightarrow\left(K^{0}(I;\epsilon,h),\partial_{\epsilon}K^{0}(I;\epsilon,h),\cdots, \partial_{\epsilon}^{(d-1)}K^{0}(I;\epsilon,h)\right)
\end{align*}
is a local diffeomorphism. Moreover, we have
\begin{align*}
G_{1}|\eta(I_{1})-\eta(I_{2})|\leq|I_{1}-I_{2}|\leq G_{2}|(I_{1})-\eta(I_{2})|
\end{align*}
for some positive constants $G_{1},G_{2}$ that depend on our choice of perturbation $H$ but are uniform in $\epsilon$ and $h$.
\end{proposition}

\begin{remark}
We will write $I_{m}=h(m+\vartheta/4)$, for $m\in \mathcal{M}_{h}(\epsilon)$, so a fixed $I_{m}\in h(\mathbb{Z}^{d}+\vartheta/4),\ m\in \mathcal{M}_{h}(\epsilon)$ will only be in $E_{\gamma}(\epsilon)\subset D$ for $O(h)$-sized intervals as $\epsilon$ varies.
\end{remark}
For any two distinct $I_{m}, I_{m^{\prime}}$, we have:
\begin{proposition}\label{aj}
Suppose that
\begin{align*}
 \lim_{h\rightarrow 0}\frac{C_{1}\gamma h^{-1/2}}{( \Delta^{-1}( C_{1}h^{-1/2}))^{2}}=\infty.
\end{align*}
Then, for all distinct $m,\ m^{\prime}\in \mathbb{Z}^{d}$ with $I_{m},\ I_{m^{\prime}}\in D$ such that
\begin{align*}
|I_{m}-I_{m^{\prime}}|\leq h \Delta^{-1}( C_{1}h^{-1/2}),
\end{align*}
and $m\in \mathcal{M}_{h}(\epsilon)$, we have
\begin{align*}
| \mu_{m}-\mu_{m^{\prime}}|\geq C_{2}h^{3/2},
\end{align*}
where the constants $C_{1},C_{2}$ depend on the choice of perturbation $H$ and on the non-resonant constant $\kappa$ but independent of $t$ and $h$.
\end{proposition}
\begin{proof}
Firstly, we prove the first estimate:
\begin{align}\label{q}
|I_{m}-I_{m^{\prime}}|\leq h \Delta^{-1}( C_{1}h^{-1/2})\quad \text{if and only if}\quad |m-m^{\prime}|\leq \Delta^{-1}( C_{1}h^{-1/2}).
\end{align}
Next, we already know that $K^{0}$ has a semiclassical expansion, that is,
\begin{align*}
K^{0}(I;\epsilon,h)=\sum_{0\leq j\leq\eta h^{-1/\rho}}K_{j}(I;\epsilon)h^{j},
\end{align*}
and following from the leading order term in the semiclassical expansion of $K^{0}$, we have
\begin{align*}
|\mu_{m}-\mu_{m^{\prime}}|\geq |K_{0}(I_{m},\epsilon)-K_{0}(I_{m^{\prime}},\epsilon)|+O(h^{2})
\end{align*}
uniformly for $\epsilon<\epsilon^{*}$. The function $K_{0}(I_{m^{\prime}},\epsilon)$  undergoes a Taylor expansion with respect to the variable $I$ at $I_{m}$ to second order:
\begin{align}\label{r}
K_{0}(I_{m^{\prime}},\epsilon)=K_{0}(I_{m},\epsilon)+\langle \nabla_{I}K_{0}(I_{m},\epsilon),(I_{n}-I_{m})\rangle+\langle \nabla_{I}^{2}K_{0}(\hat{I},\epsilon)(I_{m^{\prime}}-I_{m}),(I_{m^{\prime}}-I_{m})\rangle
\end{align}
for some $\hat{I}$ on the line segment between $I_{m^{\prime}}$ and $I_{m}$ in components.
\par
Since $m\in \mathcal{M}_{h}(\epsilon)$, we also have
\begin{align*}
|\nabla_{I}K_{0}(I_{m},\epsilon)-\nabla_{I}K_{0}(I_{\omega},\epsilon)|=O(h)
\end{align*}
uniformly to $\epsilon<\epsilon^{*}$, where $I_{\omega}$ is some non-resonant action corresponding to a non-resonant frequency $\omega\in \Omega_{0}$. Bring this estimate into \eqref{r}, and using the fact $\omega=\nabla_{I}K_{0}(I_{\omega},\epsilon),\ I_{m^{\prime}}-I_{m}=h(m^{\prime}-m)$, we get
\begin{align}\label{s}
|\mu_{m}-\mu_{m^{\prime}}|&\geq |K_{0}(I_{m},\epsilon)-K_{0}(I_{m^{\prime}},\epsilon)|\\ \notag
& =h\nabla_{I}K_{0}(I_{\omega},\epsilon)\cdot(m-m^{\prime})+O(h^{2}|m-m^{\prime}|^{2})+O(h^{2}|m-m^{\prime}|)\\ \notag
&  =h\nabla_{I}K_{0}(I_{\omega},\epsilon)\cdot(m-m^{\prime})+O(h^{2}|m-m^{\prime}|^{2})\\ \notag
& =h\omega\cdot(m-m^{\prime})+O(h^{2}|m-m^{\prime}|^{2})\\ \notag
& \geq \frac{h\gamma}{\Delta(|m-m^{\prime}|)}+O(h^{2}|m-m^{\prime}|^{2})\\ \notag
& \geq \kappa C_{1} h^{3/2}+O((h \Delta^{-1}( C_{1}h^{-1/2}))^{2}).
\end{align}
For the assumption:
\begin{align}
 \lim_{h\rightarrow 0}\frac{C_{1}\gamma h^{-1/2}}{( \Delta^{-1}( C_{1}h^{-1/2}))^{2}}=\infty,
\end{align}
we demand that  $\Delta^{-1}( C_{1}h^{-1/2}))\rightarrow\infty,$ as $h\rightarrow0$, and slower than $h^{1/4}\rightarrow0$. We can infer that $h^{3/2}$ is lower-order term. This claim is proved upon choosing $C_{1}$ appropriately.
\end{proof}
\begin{remark}
It is worth noting that the inverse of Proposition \ref{aj} is true. In a word, if two distinct quasi-eigenvalues $\mu_{m},\ \mu_{m^{\prime}}$ are very close (even less than $C_{2}h^{3/2}$ apart), then their actions $I_{m},\ I_{m^{\prime}}$ are at least $h\Delta^{-1}(C_{1}h^{-1/2})$ distance apart.
\end{remark}

By applying Proposition \ref{am}, we obtain a lower bound on the difference of speeds $\partial_{\epsilon}(\mu_{m}-\mu_{m^{\prime}})$. Consequently, this implies that two quantities diverge rapidly as $\epsilon$ evolves. This behavior can be quantified as follows:
\begin{proposition}\label{ak}
Choose any $\delta>7/4$. Suppose that $h<h_{0},\ m,\ m^{\prime}\in \mathbb{Z}^{d}$ are distinct, and  $\epsilon_{0}\in (0,\epsilon^{*})$ are fixed with $I_{m},\ I_{m^{\prime}}\in D,\ m\in \mathcal{M}_{h}(\epsilon)$ and
\begin{align}\label{t}
|\ \mu_{m}(\epsilon_{0},h)-\mu_{m^{\prime}}(\epsilon_{0},h)|<h^{\delta}<h^{7/4}.
\end{align}
Set
\begin{align*}
\mathcal{C}_{m,m^{\prime}}(h)=\{\epsilon\in (0,\epsilon^{*}):|\ \mu_{m}(\epsilon,h)-\mu_{m^{\prime}}(\epsilon,h)|<h^{\delta}\}.
\end{align*}
Then there exist positive constants $\tilde{C}_{1},\tilde{C}_{2}$ which depend on the constants $C_{1},\ C_{2}$ from Proposition \ref{aj} as well as the geometric constants $G_{1},\ G_{2}$ from Proposition \ref{am} such that
\begin{align}\label{bf}
\frac{meas\left([\epsilon^{*}-\tilde{C}_{1}h^{3/4},\epsilon^{*}+\tilde{C}_{1}h^{3/4}]\cap\mathcal{C}_{m,n}\right)}{h^{3/4}}<\tilde{C}_{2}h^{\delta-7/4}
(\Delta^{-1}(C_{1}h^{-1/2}))^{-1}.
\end{align}
Moreover, when $\delta>7/4+d$, we also have the following estimate
\begin{align*}
meas\left(\left\{t\in(0,\epsilon^{*}):m\ \in\ \mathcal{M}_{h}(t)\ and\ |\mu_{m}-\mu_{m^{\prime}}|  < h^{\gamma}\ for\ all\ m^{\prime}\ \neq\ m, I_{m^{\prime}}\ \in\ D\right\}\right)\\
=O(h^{\delta-7/4-d}(\Delta^{-1}(C_{1}h^{-1/2}))^{-1}).
\end{align*}
\end{proposition}
\begin{proof}
Applying the inverse of Proposition \ref{aj}, we have $|I_{m}-I_{m^{\prime}}|\geq h\Delta^{-1}(C_{1}h^{-1/2})$ and as an application of Proposition \ref{am}, we obtain
\begin{align*}
|\partial_{\epsilon}\mu_{m}(\epsilon;h)-\partial_{\epsilon}\mu_{m^{\prime}}(\epsilon;h)| = |\partial_{\epsilon}K^{0}(I_{m},\epsilon;h)-\partial_{\epsilon}K^{0}(I_{m^{\prime}},\epsilon;h)\geq C\ h\Delta^{-1}(C_{1}h^{-1/2}),
\end{align*}
where $C$ relies on $C_{1},\ C_{2}$ and the geometric constants $G_{i},i=1,2$. We do Taylor's expansion of the time parameter $\epsilon$ to the second order, then we have
\begin{align*}
\mu_{m}(\epsilon;h)=K^{0}(I_{m},\epsilon;h)=K^{0}(I_{m},\epsilon_{0};h)+(\epsilon-\epsilon_{0})\partial_{\epsilon}(K^{0}(I_{m},\epsilon;h))+O(|\epsilon-\epsilon_{0}|^{2}).
\end{align*}
In a similar way, by Taylor's expansion, we also have
\begin{align*}
\mu_{m^{\prime}}(\epsilon;h)=K^{0}(I_{m^{\prime}},\epsilon;h)=K^{0}(I_{m^{\prime}},\epsilon_{0};h)+(\epsilon-\epsilon_{0})\partial_{\epsilon}(K^{0}(I_{m^{\prime}},\epsilon;h))
+O(|\epsilon-\epsilon_{0}|^{2}),
\end{align*}
and their error terms are uniform in $h$ and $m$. Make the difference between the above two formulas, and use formula \eqref{t}:
\begin{align*}
|\ \mu_{m}(\epsilon;h)-\mu_{m^{\prime}}(\epsilon;h)|&=(\epsilon-\epsilon_{0})|\partial_{\epsilon}(K^{0}(I_{m},\epsilon;h))-\partial_{\epsilon}(K^{0}(I_{m^{\prime}},\epsilon;h))|\\
&\quad+O(h^{\gamma})+O(|\epsilon-\epsilon_{0}|^{2}).
\end{align*}
To get the results we expect, we choose $\tilde{C}_{1}$, such that the linear term for $|\epsilon-\epsilon_{0}|<\tilde{C}_{1}h^{3/4}$ can controlled the quadratic term $O(|\epsilon-\epsilon_{0}|^{2})$. Moreover, due to $\delta>7/4$, $h^{\delta}$ is controlled by the others for sufficiently small $h$. It follows that
\begin{align*}
|\ \mu_{m}(\epsilon;h)-\mu_{m^{\prime}}(\epsilon;h)| & \geq \frac{1}{2}(\epsilon-\epsilon_{0})|\partial_{\epsilon}(K^{0}(I_{m},\epsilon;h))-\partial_{\epsilon}(K^{0}(I_{m^{\prime}},\epsilon;h))| \\
&\geq \frac{1}{2}C(\epsilon-\epsilon_{0})\tilde{C}_{1}h\Delta^{-1}(C_{1}h^{-1/2})
\end{align*}
for $\epsilon\in [\epsilon_{0}-\tilde{C}_{1}h^{3/4},\epsilon_{0}+\tilde{C}_{1}h^{3/4}]$. Hence, we have
\begin{align*}
\epsilon\in\mathcal{C}_{m,m^{\prime}} & \Longleftrightarrow \frac{1}{2}C(\epsilon-\epsilon_{0})\tilde{C}_{1}h\Delta^{-1}(C_{1}h^{-1/2})\leq h^{\delta}\\
& \Longleftrightarrow|\epsilon-\epsilon_{0}|\leq C^{-1}h^{\delta-1}(\Delta^{-1}(C_{1}h^{-1/2}))^{-1}.
\end{align*}
This yields the estimate of this proposition, where the constants $\tilde{C}_{i}$ depend on the original Hamiltonian, the  perturbation, $\gamma$, and $G_{i}$, but not on $t$ or $h$. Finally, set
\begin{align*}
A_{m}=\{\epsilon\in(0,\epsilon^{*}):m\ \in\ \mathcal{M}_{h}(t)\ and\ |\ \mu_{m}-\mu_{m^{\prime}}| < h^{\delta}\ for\ all\ m^{\prime}\ \neq\ m, I_{m^{\prime}}\ \in\ D\}.
\end{align*}
By recalling a measure-theoretic lemma and using estimate \eqref{bf}, we get
\begin{align*}
meas(\{\epsilon\in(0,\epsilon^{*}):m\ \in\ \mathcal{M}_{h}(\epsilon)\ and\ |\ \mu_{m}-\mu_{m^{\prime}}| < h^{\delta}\})\\=O(h^{\delta-7/4}(\Delta^{-1}(C_{1}h^{-1/2}))^{-1}),
\end{align*}
then by summing over all such $m^{\prime}$, $\{m^{\prime}\in\mathbb{Z}^{d}:\ I_{m^{\prime}}\in D\}$, we have
\begin{align*}
meas(A_{m})&=h^{-d}\cdot meas(D)\cdot O(h^{\delta-7/4}\Delta^{-1}(C_{1}h^{-1/2}))\\&=O(h^{\delta-7/4-d}(\Delta^{-1}(C_{1}h^{-1/2}))^{-1}).
\end{align*}
This completes the proof.
\end{proof}
From Proposition \ref{ak}, we can draw a conclusion that $\forall \epsilon\in(0,\epsilon^{*})$, we have $|\mu_{m}-\mu_{m^{\prime}}|>h^{\delta}$, for $m\in \mathcal{M}_{h}(\epsilon)$ with $m^{\prime}\in\mathbb{Z}^{d}, m^{\prime}\neq m,\ I_{m},I_{m^{\prime}}\in D$. And, if we introduce so-called energy window $\mathcal{W}_{m}(h):=[\mu_{m}(h)-\frac{h^{\delta}}{3},\mu_{m}(h)+\frac{h^{\delta}}{3}]$, then $\mu_{m}$ is only quasieigenvalue in $\mathcal{W}_{m}(h),\ m\in \mathcal{M}_{h}(\epsilon),\ \epsilon\in(0,\epsilon^{*})$.
\par
\begin{remark}
In the definition of $A_{m}$, it is not required that $m^{\prime}\in\mathcal{M}_{h}(\epsilon)$.
\end{remark}
To formulate this result more precisely, we introduce a new notation. We consider a new set:
\begin{align*}
N(\epsilon;h)=\sharp \mathcal{M}_{h}(\epsilon)=\sharp \{m\in\mathbb{Z}^{d}:I_{m}\in D\ and\ \epsilon\in A_{m}\}.
\end{align*}
\begin{proposition}\label{ap}
Let $N$ be defined as above. Then for $\delta>7/4+2d$, the set
\begin{align*}
\mathfrak{g}:=\{\epsilon\in(0,\epsilon^{*}),\exists\ sequence\ h_{j}\ \rightarrow\ 0\ such\ that\ N(\epsilon;h_{j})=0\ for\ all\ j\}
\end{align*}
has full measure in $(0,\epsilon^{*})$.
\end{proposition}
The detailed proof can be found in Reference \cite{MR4404789}.
\subsection{Positive mass}
The aim of this subsection is to study the number of eigenvalue $E_{k}(h_{j})$ lying in the window $\mathcal{W}_{m}(h):=[\mu_{m}(h)-\frac{h^{\delta}}{3},\mu_{m}(h)+\frac{h^{\delta}}{3}],m\in\mathcal{M}_{h}(\epsilon)$. According to Proposition \ref{ak}, the energy windows $\mathcal{W}_{m}(h_{j})$ are disjoint. Having established that $\mathfrak{g}$ is of full measure in Proposition \ref{ap} provided $\delta>7/4+2d$, we now fix $\delta>7/4+2d$. For simplicity, we will no longer writing time $t$.
\par
Let
\begin{align*}
\mathcal{N}_{m}(h)=\sharp\{E_{k}(h)\in \mathcal{W}_{m}(h)\}.
\end{align*}
According to $Weyl's$ law, the total number of eigenvalue $E_{k}(h_{j})$ lying in the energy band $[a,b]$ is asymptotically equal to
$(2\pi h)^{-d}meas(p^{-1}([a,b]))$, where $p=\sigma(P_{h})$, meanwhile, the number of quasi-modes in our local patch $\mathbb{T}^{d}\times D$ is asymptotically equivalent to $ (2\pi h_{j})^{-d}meas(E_{\gamma})$.
\par
Fixing $\lambda>1$, we define
\begin{align}\label{w}
\tilde{\mathcal{M}}_{h}(\lambda):=\left\{m\in\mathcal{M}_{h}:\mathcal{N}_{m}(h)<\lambda \frac{meas(p^{-1}([a,b]))}{meas(E_{\gamma})}\right\}.
\end{align}
\begin{lemma}\label{ar}
For fixed $\delta>7/4+2d$ and each $m\in \tilde{\mathcal{M}}_{h_{j}}(\lambda)$, there exists an $L^{2}$-normalised eigenfunction $u_{k_{j}}(h_{j})$ with eigenvalue $E_{k_{j}}(h_{j})\in [\mu_{m}-h^{\delta}_{j}/3,\mu_{m}+h^{\delta}_{j}/3]$ such that
\begin{align}
\max_{|E_{k_{j}}-\mu_{m}|\leq h^{\gamma}_{j}}|\langle u_{k_{j}},v_{m}\rangle|\geq\frac{1-o(1)}{\lambda }\cdot\frac{meas(E_{\gamma})}{meas(p^{-1}([a,b])}.
\end{align}
\end{lemma}
The proof of this lemma can be found in \cite{MR4404789}.
We introduce the notation
\[
\begin{aligned}
\tilde{\mathcal{J}}_{h_{j}}(\lambda)&=\{h_{j}(m+\vartheta/4): m\in\tilde{\mathcal{M}}_{h_{j}}(\lambda)\},\\
\mathcal{J}_{h_{j}}(\lambda)&=\{h_{j}(m+\vartheta/4): m\in\mathcal{M}_{h_{j}}\}.
\end{aligned}
\]
We will prove that the distance between most resonant actions $E_{\gamma}$ of KAM tori and actions in $\tilde{\mathcal{M}}_{h_{j}}(\lambda)$ is $O(h_{j})$ for all sufficiently large $j$. This shows that the set $\tilde{\mathcal{M}}_{h}(\lambda)$ consisting in index corresponding the concentrating quasimodes associated to such torus actions is contained in $\mathcal{M}_{h_{j}}$, that is $\tilde{\mathcal{M}}_{h}(\lambda)\subset \mathcal{M}_{h_{j}}$.
\begin{proposition}\label{aq}
Let $\tilde{\mathcal{J}}_{h_{j}}$ be defined as above. Then we have
\begin{align}\label{y}
\frac{meas(\{I\in E_{\gamma}:dist(I,\tilde{\mathcal{J}}_{h_{j}}(\lambda))<Lh_{j}\})}{meas(E_{\gamma})}\geq1-\frac{L^{d}}{\pi^\frac{d}{2}\lambda}
\end{align}
for all sufficiently large $j$.
\end{proposition}
\begin{proof}
To prove estimate \eqref{y}, we have to use other estimates. The Pigeonhole Principle shows that for large $\lambda,\ \mathcal{N}_{m}(h)$ is only rarely larger than $\lambda\cdot \frac{meas(p^{-1}([a,b]))}{meas(E_{\gamma})}$. Indeed, the disjointness of the $\mathcal{W}_{m}(h_{j})$ implies
\begin{align*}
\limsup_{j\rightarrow\infty}h_{j}^{d}\ \sharp(\mathcal{M}_{h_{j}}\setminus\tilde{\mathcal{M}}_{h_{j}})\cdot\lambda \frac{meas(p^{-1}([a,b]))}{meas(E_{\gamma})} \leq \frac{meas(p^{-1}([a,b]))}{4\pi^{2}}.
\end{align*}
Due to $h_{j}^{d}\ \sharp\mathcal{M}_{h_{j}}\rightarrow (2\pi)^{-d}meas(E_{\gamma})$ as $j\rightarrow\infty$, we come to the conclusion
\begin{align}\label{x}
\sharp(\mathcal{M}_{h_{j}}\backslash\tilde{\mathcal{M}_{h_{j}}})<\frac{2}{\lambda}\cdot (2\pi h_{j})^{-d}meas(E_{\gamma}),
\end{align}
as desired . Then by formula \eqref{x}, we obtain
\begin{align*}
& meas(\{I\in E_{\gamma}:dist(I,\tilde{\mathcal{J}}_{h_{j}})(\lambda)<Lh_{j}\})\\ \notag
 &\quad\geq  1-meas(\{I\in E_{\gamma}:dist(I,\mathcal{J}_{h_{j}}\backslash\tilde{\mathcal{J}}_{h_{j}})(\lambda)<Lh_{j}\})\\ \notag
&\quad\geq  1-\frac{1}{meas(E_{\gamma})}\cdot\sharp(\mathcal{M}_{h_{j}}\backslash\tilde{\mathcal{M}}_{h_{j}})(\lambda))\cdot\pi^{\frac{d}{2}} (Lh_{j})^{d}\\ \notag
&\quad\geq  1-\frac{1}{meas(E_{\gamma})}\cdot\frac{2}{\lambda}(2\pi h_{j})^{-2} meas(E_{\gamma})\cdot\pi^{\frac{d}{2}}(Lh_{j})^{d}\\ \notag
&\quad=  1-\frac{L^{d}}{\pi^{\frac{d}{2}}\lambda}
\end{align*}
for all sufficiently large $j$.
\end{proof}
Due to the proof process of Proposition \ref{aq}, we could make a claim:
\begin{remark}
After exactly calculation, we get
\begin{align*}
\frac{\sharp\tilde{\mathcal{M}}_{h_{j}}}{\sharp\mathcal{M}_{h_{j}}}=1-\frac{\sharp(\mathcal{M}_{h_{j}}\setminus\tilde{\mathcal{M}}_{h_{j}})}{\sharp\mathcal{M}_{h_{j}}}>1-
\frac{2}{\lambda}
\end{align*}
for each sufficiently large $j$ by \eqref{x}. In other words, the proportion of $O(h^{\delta})$-sized energy windows associated to actions in $\mathcal{M}_{h}$ containing at most $\lambda R$ eigenvalue is at least $1-\frac{2}{\lambda}$, when $\lambda>2$.
\end{remark}
\subsection{Proof of scarring}
The key ingredient of proving scarring is Proposition \ref{aq} and Lemma \ref{ar}. The idea of this proof is similar to \cite{MR4404789} (or \cite{c}).
\begin{proof}
Take the subset of $E_{\gamma}$ in \eqref{y} denoted as  $E_{\gamma,j}(\lambda)$, and introduce a new notation as follows:
\begin{align*}
\tilde{E}_{\gamma}(\lambda):=\bigcap_{l=1}^{\infty}\bigcup_{j=l}^{\infty}E_{\gamma,j}(\lambda).
\end{align*}
Then, form Proposition \ref{aq}, we can conclude that $\tilde{E}_{\gamma}(\lambda)$ has measure at least $1-O(\lambda^{-1})$, and  for any $I\in \tilde{E}_{\gamma}(\lambda)$, we may derive that
\begin{align*}
dist(I,\tilde{\mathcal{J}}_{h_{j}}(\lambda))<Lh_{j}
\end{align*}
with infinitely many $j$. For each $I\in \tilde{E}_{\gamma}(\lambda)$ and each such $j$, we choose such an action in $\tilde{\mathcal{J}}_{h_{j}}(\lambda)$, and an associated quasimode $v_{m_{j}}$ for $P_{h_{j}}$, in order to obtain a sequence of quasimodes that concentrates completely on the torus $\Lambda_{\omega}=\{I_{\omega}\}\times\mathbb{T}^{d}$.
\par
For this sequence, we can look for a corresponding sequence of eigenfunction $u_{k_{j}}$ for  $P_{h_{j}}$ by making use of Lemma \ref{ar} such that
\begin{align}\label{u1}
|\langle u_{k_{j}}(h_{j}),v_{m_{j}}(h_{j})\rangle|>\frac{1}{2\lambda}\cdot\frac{meas(E_{\gamma})}{meas(p^{-1}([a,b]))}
\end{align}
for all sufficiently large $j$.
\par
We now claim that sequence $u_{k_{j}}(h_{j})$ scars on the torus $\Lambda_{\omega}$. This is because we can take an arbitrary semiclassical pseudo-differential operator $A_{h}$ with the symbol equal to 1 and compactly supported  around the torus $\Lambda_{\omega}$, and estimate
\begin{align*}
\langle A_{h_{j}}^{2}u_{k_{j}}(h_{j}),u_{k_{j}}(h_{j})\rangle & =\|A_{h_{j}}u_{k_{j}}(h_{j})\|^{2}\\ \notag
& \geq|\langle A_{h_{j}}u_{k_{j}}(h_{j}),v_{m_{j}}(h_{j})\rangle|^{2}\\ \notag
& =|\langle u_{k_{j}}(h_{j}),v_{m_{j}}(h_{j})\rangle+\langle u_{k_{j}}(h_{j}),(A_{h_{j}}-Id)v_{m_{j}}(h_{j})\rangle|^{2}\\ \notag
& >\frac{1}{(2\lambda )^{2}}\cdot\left(\frac{meas(E_{\gamma})}{meas(p^{-1}([a,b]))}\right)^{2}>0
\end{align*}
for all sufficiently large $j$, by \eqref{u1}, and the fact that $v_{m_{j}}$ is a sequence of quasimodes concentrating completely on the torus $\Lambda_{\omega}=\{I_{\omega}\}\times\mathbb{T}^{d}$.
\par
Now, let $\nu$ be a semiclassical measure associated to a subsequence of the $u_{k_{j}}(h_{j})$, then we can conclude that
\begin{align*}
\int\sigma(A)d\nu
\end{align*}
is bounded below by $\frac{1}{(2\lambda )^{2}}\cdot\left(\frac{meas(E_{\kappa})}{meas(p^{-1}([a,b]))}\right)^{2}$. By taking $A$ to have shrinking support around $\Lambda_{\omega}$, we see that $\nu$ has positive mass strictly greater than $\frac{1}{(2\lambda )^{2}}\cdot\left(\frac{meas(E_{\kappa})}{meas(p^{-1}([a,b]))}\right)^{2}$ on $\Lambda_{\omega}$.
\par
Applying this argument with $\lambda\rightarrow\infty$, we establish the existence of such semiclassical measure for almost all $I_{\omega}\in E_{\gamma}$ and the proof is finished.
\end{proof}
Therefore, we can conclude that the eigenstates corresponds to the scarring of the eigenvalue $E$ on the $d$-dimensional invariant torus.
\section{Appendix}
\subsection{Define a new function}
For a given $\alpha-$approximation function $\Delta$, we define a new notation in the following way:
\begin{align*}
\Gamma_{s,n}(\eta)=\sup_{t\geq0}(1+t)^{\bar{r}}\Delta^{n}(t)e^{-\eta t^{\frac{1}{\sigma}}}.
\end{align*}
For every $\eta>0$, there exists a sequence $\{\eta\}_{\nu}, \nu\geq0$ with $\eta_{1}\geq\eta_{2}\geq\cdots>0$ such that $\sum_{\nu\geq0}\eta_{\nu}\leq\eta$. We work with the supermum in the definition of $\Gamma_{s,n}(\eta)$ under such a sequence. Then, we derive the following result:
\begin{lemma}\label{ba}
If
\begin{align*}
\frac{1}{\log\kappa}\int_{T}^{\infty}\frac{\log\Delta(t)}{t^{1+\frac{1}{\alpha}}}{\rm d}t\leq a
\end{align*}
for $1<\kappa\leq2$ and $T\geq\varsigma$, then
\begin{align*}
\Gamma_{r,n}(\eta)\leq e^{\eta T^{\frac{1}{\alpha}}}.
\end{align*}
\end{lemma}
\begin{proof}
Let $\delta_{\bar{r},n}(t)=\log(1+t)^{\bar{r}}\Delta^{n}(t)$, and
\begin{align*}
\eta_{\nu}=\delta_{\bar{r},n}(t_{\nu})/t_{\nu}^{\frac{1}{\alpha}},\ t_{\nu}=\kappa^{\nu}T
\end{align*}
for any $\nu\geq0$. According to the logarithmic laws, we have
\begin{align*}
\delta_{\bar{r},n}(t)=\log(1+t)^{\bar{r}}\Delta^{n}(t)=\bar{r}\log(1+t)+n\log\Delta(t).
\end{align*}
Obviously, the function $(1+t)$ naturally is the $\alpha-$approximation function.
We choose two positive sequences, which are respectively $\{a_{\nu}\}_{\nu\geq0},\{c_{\nu}\}_{\nu\geq0}$ in order that
\begin{align*}
\eta_{\nu}=\delta_{\bar{r},n}(t_{\nu})/t_{\nu}=n\log\Delta(t_{\nu})/t^{\frac{1}{\alpha}}_{\nu}+\bar{r}\log(1+t)/t^{\frac{1}{\alpha}}_{\nu}=na_{\nu}+\bar{r}c_{\nu},
\end{align*}
where $\eta_{1}\geq\eta_{2}\geq\cdots>0$ and $\sum_{\nu\geq0}\eta_{\nu}\leq\eta$.
Due to the constraint \eqref{ac}, this provides
\begin{align*}
\sum_{\nu\geq0}a_{\nu}\leq\int_{0}^{\infty}\frac{\log\Delta(t_{\nu})}{t^{\frac{1}{\alpha}}_{\nu}}
{\rm d}\nu\leq\frac{1}{\log\kappa}\int_{T}^{\infty}\frac{\log\Delta(t)}
{t^{1+\frac{1}{\alpha}}}{\rm d}t\leq a.
\end{align*}
In the same way,
\begin{align*}
\sum_{\nu\geq0}c_{\nu}\leq\int_{0}^{\infty}\frac{\log(1+t_{\nu})}{t^{\frac{1}{\alpha}}_{\nu}}{\rm d}\nu
\leq\frac{1}{\log\kappa}\int_{T}^{\infty}\frac{\log(1+t)}{t^{1+\frac{1}{\alpha}}}{\rm d}t
\leq c.
\end{align*}
Then
\begin{align*}
\sum_{\nu\geq0}\eta_{\nu}\leq\int_{0}^{\infty}\frac{\delta_{s}(t_{\nu})}{t^{\frac{1}{\alpha}}_{\nu}}{\rm d}\nu
\leq\frac{1}{\log\kappa}\int_{T}^{\infty}\frac{\delta_{s}(t)}
{t^{1+\frac{1}{\alpha}}}{\rm d}t
\leq \eta
\end{align*}
with $\eta=na+\bar{r}c$.
Since $\delta_{\bar{r},n}(t)-\eta_{\nu} t^{\frac{1}{\alpha}}\leq0$ for $t\geq t_{\nu}$ by monotonicity \eqref{ab}, and on the interval $[0,t_{\nu}]$, the supremum is achieved and smaller than $\delta_{\bar{r},n}(t_{\nu})$. It follows that
\begin{align*}
\Gamma_{\bar{r},n}(\eta_{\nu})=\sup_{t\geq0}e^{\bar{r}\log(1+t)+n\log\Delta(t)-\eta_{\nu}t^{\frac{1}{\alpha}}}\leq e^{\delta_{\bar{r},n}({t_{\nu}})}=e^{\eta_{\nu}t_{\nu}^{\frac{1}{\alpha}}}
\end{align*}
by the definition of $\eta_{\nu}$. The function $\Gamma_{s}(\eta)$ is monotonically decreasing with respect to $\eta$, so we conclude that
\begin{align*}
\Gamma_{\bar{r},n}(\eta)\leq\Gamma_{\bar{r},n}(\eta_{\nu})\leq e^{\eta_{\nu}t_{\nu}^{\frac{1}{\alpha}}}\leq e^{\eta t_{\nu}^{\frac{1}{\alpha}}}
\end{align*}
for any $\nu\geq0$. For convenience, we take $\nu=0$, that is,
\begin{align*}
\Gamma_{\bar{r},n}(\eta)\leq\Gamma_{\bar{r},n}(\eta_{\nu})\leq e^{\eta_{\nu}t_{\nu}^{\frac{1}{\alpha}}}\leq e^{\eta T^{\frac{1}{\alpha}}}
\end{align*}
with $\eta=na+c\bar{r}$.
\end{proof}
The result here is a generalization of the lemma 2.1 result in reference \cite{a}.
\section*{Acknowledgments}
The second author (Y. Li) was supported by  National Natural Science Foundation of China (12071175 and 12471183).

\end{document}